\documentclass[12pt, a4paper]{amsart} 
\usepackage[margin=0.7in]{geometry}
\usepackage[utf8]{inputenc}
\usepackage{verbatim}
\usepackage{lmodern}
\usepackage{textcomp}
\usepackage{color}
\usepackage{amssymb,amscd,amsthm,amsxtra,mathtools}
\usepackage{amsmath}
\usepackage{dsfont,latexsym, color}
\usepackage{mathrsfs}
\usepackage{fancyhdr, enumitem}
\usepackage{bm}
\usepackage[alphabetic,initials]{amsrefs}
\usepackage{color}
\definecolor{citation}{rgb}{0.2,0.58,0.2} 
\definecolor{formula}{rgb}{0.1,0.2,0.6}
\definecolor{url}{rgb}{0.3,0,0.5} 
\usepackage[colorlinks=true,linkcolor=formula,urlcolor=url,citecolor=citation]{hyperref} 

\allowdisplaybreaks

\newtheorem{theorem}{Theorem}[section]
\newtheorem{prop}[theorem]{Proposition}
\newtheorem{corol}[theorem]{Corollary}

\theoremstyle{remark}
\newtheorem{rem}[theorem]{Remark}
 
\numberwithin{equation}{section}
\allowdisplaybreaks[4]

\def\XXint#1#2#3{{\setbox0=\hbox{$#1{#2#3}{\int}$}
		\vcenter{\hbox{$#2#3$}}\kern-.5\wd0}}

\newcommand{\R}{\mathbb{R}}
\newcommand{\N}{\mathbb{N}}

\renewcommand{\Delta}{\varDelta}
\renewcommand{\Sigma}{\varSigma}
\renewcommand{\Lambda}{\varLambda}
\renewcommand{\Psi}{\varPsi}
\renewcommand{\Omega}{\varOmega}
\renewcommand{\Gamma}{\varGamma}

\renewcommand{\epsilon}{\varepsilon}

\begin{document}

\title[Double well potentials from transition layers]{Reconstructing double-well potentials \\ from transition layers \\ in long-range phase coexistence models}
\author[F.~De~Pas]{Francesco De Pas} \address{Francesco De Pas\\  Department of Mathematics and Statistics\\
	University of Western Australia,\\
	35 Stirling Highway, WA 6009 Crawley, Australia}
	\email{\href{mailto:francesco.depas@uwa.edu.au}{francesco.depas@uwa.edu.au}}
\author[S.~Dipierro]{Serena Dipierro}  \address{Serena Dipierro\\  Department of Mathematics and Statistics\\
	University of Western Australia,\\
	35 Stirling Highway, WA 6009 Crawley, Australia}
	\email{\href{mailto:serena.dipierro@uwa.edu.au}{serena.dipierro@uwa.edu.au}}
\author[E. Valdinoci]{Enrico Valdinoci}  \address{Enrico Valdinoci\\Department of Mathematics and Statistics\\
	University of Western Australia,\\
	35 Stirling Highway, WA 6009 Crawley, Australia}
\email{\href{mailto:enrico.valdinoci@uwa.edu.au}{enrico.valdinoci@uwa.edu.au}}

	

\thanks{{\it Aknowledgements.} 
SD and EV are members of the Australian Mathematical Society.
FDP and EV are supported by the Australian Laureate Fellowship FL190100081 ``Minimal
surfaces, free boundaries and partial differential equations''.
SD is supported by the Australian Future Fellowship
FT230100333 ``New perspectives on nonlocal equations''.
}

\begin{abstract}
In models of phase coexistence, the precise form of the double-well potential is of central importance, yet it cannot be derived from first principles.  

In this paper, we investigate an inverse problem: starting from a prescribed transition layer with power-type decay at infinity, we reconstruct the structural properties of the associated double-well potential.  
We focus on the case of long-range interactions, where the dependence of the potential on the layer and its derivatives is particularly delicate.  

Our analysis establishes a correspondence between the decay rate of the transition layer and the regularity of the potential, revealing the existence of specific patterns and the possible emergence of degeneracies.

\end{abstract}

\maketitle

\setcounter{equation}{0}\setcounter{theorem}{0}

\setcounter{tocdepth}{1} 
\begin{center}
	\begin{minipage}{11cm}
\footnotesize
		\tableofcontents
	\end{minipage}
\end{center}

\section{Introduction}

In this work we study energy functionals modeling phase coexistence with long-range particle interactions, of the form
\begin{equation}\label{main_fun}
 \frac{1}{4} \iint_{\R^{2n}\setminus (\R^n \setminus \Omega)^2} \frac{|u(x)-u(y)|^2}{|x-y|^{n+2s}} \,dx\,dy + \int_{\R^n} W(u(x)) \, dx,
\end{equation}
where~$s\in (0,1)$ and~$\Omega \subset \R^n$ is an open set.  
The first term in~\eqref{main_fun} encodes nonlocal particle interactions, while~$W$ is a double-well potential.
The pure phases of the model are normalized to be~$-1$ and~$+1$, which correspond to the minima of the potential~$W$
(hence, up to a vertical translation, we can suppose that~$W(r)>0=W(-1)=W(+1)$ for all~$r\in(-1,1)$).
\medskip
  
Such functionals are a non-scaled, fractional analogue of the classical Ginzburg--Landau energy, in which the Dirichlet term is replaced by a fractional Sobolev seminorm. Traditionally, one is interested in the equilibrium configurations of the system, which correspond to the solutions of the Euler--Lagrange equation
\begin{equation}\label{EL_eq}
L_s u(x) = W'(u(x)),
\end{equation}
where~$L_s$ denotes the fractional Laplacian
\begin{equation}\label{fralapc}
L_s u(x) := \text{\rm PV}_x \int_{\R^n} \frac{u(y)-u(x)}{|x-y|^{n+2s}}\, dy,
\end{equation}
with~${\rm PV}$ denoting the Cauchy principal value. 

The equation in~\eqref{EL_eq} is often referred to in the literature as the 
\textit{fractional Allen--Cahn equation} and constitutes a paradigmatic model 
for nonlocal phase transitions. 
Energies as in~\eqref{main_fun} naturally arise when studying phase transition
processes with long-range tension effects (see e.g.~\cite{PSV13, CS14, MR3280032, CozziValdNONLINEARITY, MR4581189}) 
and in the Peierls--Nabarro model for crystal dislocations 
(see e.g.~\cite{MR371203, MR1442163, MR2461827, MR2946964, GM12, DFV14, DPV15, MR3338445, BV16, MR3511786, MR3703556, MR4612096, MR4531940}). 

Among all the equilibria of~\eqref{main_fun}, one is usually interested in the local minima, due to their typically enhanced stability properties.
The nontrivial minimizers, often called \emph{transition layers}, interpolate between the two stable states of~$W$ and are known to display rich qualitative features: symmetry, monotonicity, rigidity phenomena, and connections with nonlocal minimal surfaces (see e.g.~\cite{MR2177165, MR2498561, MR2644786, SV12, MR3035063, MR3148114, SV14, BV16, MR3740395, MR3812860, MR3939768, MR4124116, MR4050103, MR4938046}).  
Because of these properties, the detailed study of such minimizers is a major theme in the analysis of (classical and) nonlocal~PDEs.

Traditionally, one considers a potential~$W$---often chosen from heuristic or phenomenological considerations---and studies the existence, regularity, and qualitative features of the corresponding transition layers.

In this paper we take the opposite viewpoint: starting from a given power-like transition profile, we \emph{reconstruct} a potential that admits it as a stationary solution.  
More precisely, we prescribe an increasing function~${\phi}:\R\to(-1,1)$ connecting the two stable states~$\pm 1$ and approaching them at infinity with a given decay rate.  
This choice reflects the qualitative behavior of typical transition layers and is in line with the barrier functions often used in the analysis of fractional Allen--Cahn equations.  
We then construct a potential~$V$ such that
\begin{equation}\label{inverse_eq}
L_s {\phi}(x) = V'({\phi}(x)),
\end{equation}
and study the regularity and structure of~$V$ in terms of the asymptotic behavior of~${\phi}$.
\medskip

Our main result characterizes~$V$ as a double-well potential and establishes a precise correspondence between the decay of~${\phi}$ at the wells and the behavior of~$V$ and its derivatives 
near~$\pm 1$. 

In particular, we show that~$V$ inherits a power-like structure 
at every derivative order. Despite the nonlocal nature of~$L_s$, with a 
suitable choice of the layer~${\phi}$ one can construct a potential 
that is smooth in~$(-1,1)$ and vanishes at the wells at any prescribed order.  
This is an important feature, because potentials with vanishing
derivatives of order higher than~$2$ correspond to ``degenerate'' wells, for which one
can expect a slower decay of the solutions towards the pure phases.

Also, we stress that the double-well nature of~$V$ is not evident from its definition as a function of a given profile~$\phi$ (see~\eqref{pmes}). While the decay of~$\phi$ ensures that~$V$ is smooth in~$(-1,1)$, its regularity at the wells~$\pm1$ may drop to only Lipschitz. For regular potentials, the double-well structure was already established in~\cite{CS14}. Our setting instead allows such limited regularity at the wells (at the price of prescribing a sufficiently regular transition layer).

From a mathematical standpoint, this type of results
relies on two notable facts. 
First, the fractional Laplacian of the transition layer enjoys good scaling properties on derivatives 
(as described in Proposition~\ref{proppo}), and this fact is nontrivial given the 
nonlocal nature of the operator\footnote{In spite of its own scale invariance,
the nonlocal character of the fractional Laplacian operator does not make it compatible, in general, with the notion of power-like functions. To see this, given~$a<b$ and~$\lambda\in\R$,
we observe that, on the one hand, if~$f(x)=x^\lambda$ for all~$x\in(a,b)$, then all the derivatives of~$f$ are power-like in~$(a,b)$. On the other hand, 
by~\cite[Theorem~1.5]{MR4297378},
for every~$m\in\N$, every~$F\in C^m(\R)$, and
every~$\epsilon>0$, one can find functions~$f_\epsilon$ and~$\eta_\epsilon$ such that~$\|f-f_\epsilon\|_{C^m((a,b))}\le\epsilon$,
$\|\eta_\epsilon\|_{L^\infty((a,b))}\le\epsilon$, and
$$(-\Delta)^s f_\epsilon=F(x)+\eta_\epsilon\qquad{\mbox{in }}(a,b).$$
For example, the fractional Laplacian of a function which ``looks like a power in $(a,b)$'',
may well ``look like an exponential in~$(a,b)$''.} and, to the best of our knowledge, has not been 
addressed in the existing literature. 
Second, when differentiating~$V'(u)$ with respect to~$x$, the terms that appear 
exhibit the same decay at infinity, thus allowing for a quantitative analysis 
of higher-order derivatives of~$V$ near the wells (see Proposition~\ref{lemm}).
\medskip

We mention here that the inverse construction that we implement in this paper 
(from the transition layer to the potential) 
is relevant both conceptually and in applications. 

From a theoretical viewpoint, our result can be seen as a continuous extension 
of the classical theory to a nonlocal framework: much as in the local 
Allen--Cahn equation, one can reconstruct the derivative of the potential from 
the second derivative of the layer and then iteratively recover higher-order 
derivatives. It is worth stressing that such a correspondence is not 
always guaranteed when passing from local to nonlocal models. Here, however, the 
long-range interactions encoded in~$L_s$ do not disrupt the mechanism.

Moreover, we highlight the possibility of 
``designing'' both degenerate and non-degenerate potentials so as to reproduce 
prescribed long-range profiles, in particular those with power-like behavior 
(see Remark~\ref{pollohay}). 
This point is significant if we consider that the potential in these models is 
usually chosen for technical convenience, as its precise analytic form can 
rarely be determined. Indeed, such models are typically not derived 
microscopically (e.g. from statistical or quantum mechanics of particles), but 
rather proposed phenomenologically, in the spirit of Landau's theory of phase 
transitions (see e.g.~\cite{MR4581189} for a modern account). 
Landau's theory postulates that near a critical point the free energy can be 
expanded as a power series in the order parameter, though the coefficients are 
not fixed by first principles. The specific form of the potential is thus 
determined by intricate microscopic interactions, chemical composition, and 
thermodynamic properties of the material, which explains why the bulk free 
energy function can vary significantly between different substances.

In this spirit, the inverse construction that we implement has useful practical consequences, as it allows, at least in principle, the reconstruction of the potential from observations of the material:
specifically, by detecting variations of the state parameter near the interface and its decay toward the pure phases.

Furthermore, from an applied viewpoint, starting from a known profile and building the corresponding potential allows for refined barrier estimates and, in some cases, for proving their optimality (see~\cite{DFV14, DPV15, DPDV, OURREC}).

These considerations also point to several questions beyond the
stationary setting considered in this paper. In particular, we are confident that the results obtained here can also shed some light on the associated parabolic equation. We plan to investigate this question in a forthcoming paper by combining two approaches. The first one, will be based on gradient flow techniques together with a {\L}ojasiewicz--Simon inequality (whose exponent needs to be related to the spectral properties of the stationary equation). The second approach, instead, will rely on the construction of suitable barriers that capture the long-time asymptotics. The literature motivating the first approach includes \cite{MR1609269, MR1829143}
(as well as~\cite{MR1680877} for the higher-order case), while the second is inspired by \cite{MR3511786, MR3703556}.

\medskip 

The rest of the paper is organized as follows. In Section~\ref{n0biob}, we introduce the main results of the paper. % Section~\ref{mainr} states the main result of the paper, namely Theorem~\ref{minth}.
Section~\ref{auxilres} collects some auxiliary lemmata and technical tools needed for our analysis. Then, in Section~\ref{minth_proof}, we provide the proof of Theorem~\ref{minth}. Finally, Sections~\ref{atanx} and~\ref{mlml} contain some comments on Theorem~\ref{minth}.

\section{Main results}\label{n0biob}

Before stating the main results of the paper, we introduce here below some notations.

Let~$s\in (0,1)$, $\alpha$, $\beta\in(0,2s]$, $\kappa\in (0,+\infty)$
and~$C_1$, $C_2 >0$. We consider a function~${{\phi}} \in C^{\infty}(\R)$ such that~${\phi}'>0$ and
\begin{equation}\label{valesem}
{{\phi}}(x):= \begin{cases}
-1+C_1|x|^{-\alpha} &\mbox{if } x<-\kappa, \\
1-C_2|x|^{-\beta} &\mbox{if } x>\kappa.
\end{cases}
\end{equation} 
This will be the prototype of layer solution considered in this article.

Our objective is to build a double-well potential~$V$ with wells at~$\pm1$. In particular, recalling the operator~$L_s$ in~\eqref{fralapc}, we define the functions
\begin{equation}\label{i3456789097654bvcxs}
\begin{split}
&g(t):= L_s{{\phi}}(t) \quad {\mbox{ for all }}t \in \R \\
{\mbox{and }}\quad &h(r):= g({{\phi}}^{-1}(r)) \quad {\mbox{for all }}r \in (-1,1)
\end{split}
\end{equation}
and the potential~$ V:[-1,1] \to \R$ as
\begin{equation}\label{pmes}
V(r) := \int_{-1}^{r}h(\rho) \, d\rho.
\end{equation}
In this way, the function~${{\phi}}$ satisfies 
\begin{equation}\label{l}
L_s{{\phi}}(x)=  V\rq{}({{\phi}}(x)) \quad\mbox{for any } x \in \R.
\end{equation}

Also, we mention that, given any function~$f$, its~$i^{\text{th}}$ derivative will be denoted either by~$f^{(i)}$ or, when convenient, using prime notation~$f', f'', \ldots$
\medskip

We are now in the position to state the main theorem of this work. It establishes that~$V$ is, in fact, double-well shaped. Also, it shows  how the decay properties of
the the power-type transition layer~${{\phi}}$ influence the behavior of the associated potential~$V$ near the wells~$\pm1$. Finally, it addresses the main regularity properties of~$V$.

%%%%We refer to the notations introduced in Section~\ref{n0biob}.

\begin{theorem}\label{minth}
It holds that
\begin{equation}\label{l4b967}
V \in C^{\infty}((-1,1))%\cap C^{2,1}([-1,1])
.\end{equation}
Furthermore, the potential~$V$ satisfies
\begin{equation}\label{48397tasvjkdfbgkewguo}
{\mbox{${V(\pm1) = 0}$ and~$V(r)>0$ for any~$r\in(-1,1)$.}}
\end{equation}

In addition,
\begin{equation}\label{ham}
\lim_{r \to -1^+} \frac{V(r)}{(1+r)^{\frac{2s}\alpha+1}}=\frac{\alpha C_1^{-\frac{2s}{\alpha}}}{(2s+\alpha)s} \qquad\mbox{\rm and}\qquad \lim_{r \to 1^-} \frac{V(r)}{(1-r)^{\frac{2s}\beta+1}} =\frac{\beta C_2^{-\frac{2s}\beta}}{(2s+\beta)s};
\end{equation}
and
\begin{equation}\label{hamm}
\lim_{r \to -1^+} \frac{V\rq{}(r)}{(1+r)^{\frac{2s}\alpha}}=\frac{C_1^{-\frac{2s}{\alpha}}}{s} \qquad\mbox{
\rm and}\qquad \lim_{r \to 1^-} \frac{V\rq{}(r)}{(1-r)^{\frac{2s}\beta}} =-\frac{C_2^{-\frac{2s}\beta}}{s}.
\end{equation}

Also, for any~$i \in \N$ we have
\begin{equation}\label{0ijuhed}
\lim_{r\to -1^+} \frac{V^{(i+1)}(r)}{(1+r)^{\frac{2s}\alpha- i}} \in (-\infty,+\infty) \qquad{\mbox{and}}\qquad \lim_{r\to 1^-} \frac{V^{(i+1)}(r)}{(1-r)^{\frac{2s}\beta- i}} \in (-\infty,+\infty).\end{equation}

In particular, if~$i\in \N\setminus\{0\}$ and~${2s}\geq{\alpha} i$, then
\begin{equation}\label{trep1}
\lim_{r \to -1^+} \frac{ {V}^{(i+1)}(r)}{(1+r)^{\frac{2s}\alpha-i}} =
\frac{1}{s} C_1^{-\frac{2s}\alpha} \prod_{j=0}^{i-1} \left( \frac{2s}\alpha -j\right).
\end{equation}

Similarly, if~$i\in \N\setminus\{0\}$ and~${2s}\geq{\beta} i$, then
\begin{equation}\label{trep2}
\lim_{r \to 1^-} \frac{ {V}^{(i+1)}(r)}{(1-r)^{\frac{2s}\beta-i}}  =
\frac{(-1)^{i+1}}{s} C_2^{-\frac{2s}\beta} \prod_{j=0}^{i-1} \left( \frac{2s}\beta -j\right).
\end{equation}

Moreover, the identity in~\eqref{trep1} holds for any~$i \in \N$ whenever~$\frac{2s}\alpha$ is not an integer. Similarly, the identity in~\eqref{trep2} holds for any~$i \in \N$ whenever~$\frac{2s}\beta$ is not an integer.

Finally,
\begin{equation}\label{REgg}
V \in C^{\lfloor \frac{2s}\alpha \rfloor,1}([-1,0]) \quad\mbox{and}\quad V \in C^{\lfloor \frac{2s}\beta \rfloor,1}([0,1]),
\end{equation} where we use the notation
$$\lfloor x\rfloor:=\max\{y\in\N\;{\mbox{ s.t. }}\; y\le x\}.$$
\end{theorem}

%\begin{rem}
%This holds for any~$k\in\N$ if we can assure~$\frac{2s}\alpha\notin \N$. Indeed, the restrictions come from the L\rq{}H\^{o}pital\rq{}s Rule requiring indeterminate form of the quotient. If, there exists~$\bar{k}$ such that~$\frac{2s}\alpha=\bar{k}$ then Theorem~\ref{minth} is valid till to the~$\bar{k}^{\rm th}$ derivative.
%\end{rem}

\begin{rem}
The properties of~$V$ in~\eqref{48397tasvjkdfbgkewguo}
confirm that~$V$ has indeed the shape of a double-well potential. Besides,
it is immediate from Theorem~\ref{minth} that~$V'$ vanishes at~$\pm 1$.
\end{rem}

\begin{rem}
The choice~$\alpha=\beta = 2s$ in~\eqref{trep1} and~\eqref{trep2} corresponds to a non-degenerate potential and leads to
\[ \lim_{r\to-1^+} V''(r)= \frac1{sC_1}\qquad{\mbox{and}}\qquad \lim_{r\to1^-} V''(r) = \frac1{sC_2}.\] 
This choice of the parameters~$\alpha$ and~$\beta$ is consistent with the construction in~\cite[Section~7]{DPV15}, as well as with the decay properties derived in~\cite{PSV13, CP16}. We point out that  transition layers decay polynomially in the presence of long-range interactions (see~\cite{PSV13, CP16})
and this is an important structural difference with respect to the local case that exhibits, when the potential is non-degenerate, an exponential decay at infinity. 

On the other hand, when~$\alpha<2s$ and~$\beta<2s$, we have that~$V''(\pm1) = 0$, meaning that the potential~$V$ is degenerate. This setting aligns with the framework considered, for instance, in~\cite{DPDV}. 

On a more general note, the limits in~\eqref{trep1} and~\eqref{trep2} ensure that, given any~$i\in \N$, it is always possible to obtain~$V \in C^{i,1}((-1,1))$ and~$V^{(i)}(\pm1)=0$ by taking~$\alpha$, $\beta\in(0,\tfrac{2s}i]$.
\end{rem}

\begin{rem}\label{pollohay}
As mentioned above, an important consequence of Theorem~\ref{minth} is that choosing a transition layer~$\phi$ of polynomial type produces a potential that preserves the characteristic features of power-type potentials. Indeed, as follows from~\eqref{trep1} and~\eqref{trep2}, each differentiation yields a precise scaling behavior, lowering the effective order by one near the wells. This property, however, relies on the specific choice of the profile~${\phi}$. If one uses layers that are not exactly of power form, but only asymptotically comparable to a power, such as ~$\arctan(x)$, then (perhaps quite surprisingly)
the situation significantly changes. 
In such cases, in fact, the resulting potential may fail to exhibit 
a power-like structure, as shown in Section~\ref{atanx}. 
\end{rem}

\begin{rem}\label{remmlml}
Suppose that $\tfrac{2s}{\beta}= i+m$ for some $i\in\N$ and $m\in[0,1)$.  
Then, according to Theorem~\ref{minth}, one has that~$V\in C^{\infty}((-1,1))$ and
\[
\lim_{r\to 1^-} \frac{V^{(i+1)}(r)}{(1-r)^{m}}
=
\frac{(-1)^{i+1}}{s}\, C_2^{-\frac{2s}{\beta}}
\prod_{j=0}^{i-1}\left(\frac{2s}{\beta}-j\right).
\]
However, we point out that these properties alone do not guarantee any H\"older continuity of~$V^{(i+1)}$, as discussed in Section~\ref{mlml}.
\end{rem}

\section{Auxiliary results}\label{auxilres}
\subsection{Asymptotic estimates on~$L_s$}\label{90fxd9916}
We collect here asymptotic results regarding the operator~$L_s$, as defined in~\eqref{fralapc}.

%%%%%%We recall the notations introduced in Section~\ref{n0biob}.

\begin{prop}\label{thone}
It holds that
\begin{equation}\label{yuuy}
\lim_{x \to -\infty} |x|^{2s} L_s {{\phi}}(x) = \frac{1}{s} \qquad\mbox{and}\qquad\lim_{x \to +\infty} |x|^{2s} L_s {{\phi}}(x) = -\frac{1}{s}.
\end{equation}
\end{prop}

\begin{proof}
We will establish the limit as~$x \to -\infty$ in~\eqref{yuuy}.
The proof for the limit as~$x\to +\infty$ is similar and therefore will be omitted.

Let~$f: (-\infty,0) \to \R$ be defined, for every~$\theta\in(-\infty,0)$, as
\begin{equation}\label{defeeffe00}
 f(\theta):= \frac{ 1-|\theta|^{-\alpha}}{|1+\theta|^{1+2s}}.\end{equation}
We observe that
\begin{equation}  \label{a}
\lim_{\theta \to -\infty} |\theta|^{1+2s}f(\theta) =1 %%\qquad\mbox{and}\qquad\lim_{\theta \to 0} |\theta|^{\alpha}f(\theta)= -1
.
\end{equation}

{N}ow, for any~$x<-2\kappa$ we set
\begin{equation}\label{tanrobb}
\begin{split}
&A(x):=- PV_x \int_{-\infty}^{-\kappa} \frac{{{\phi}}(x)-{{\phi}}(y)}{|x-y|^{1+2s}} \, dy,  \qquad B(x):=- \int_{-\kappa}^{\kappa} \frac{{{\phi}}(x)-{{\phi}}(y)}{|x-y|^{1+2s}} \, dy \\
\mbox{and} \qquad &D(x):=- \int_{\kappa}^{+\infty} \frac{{{\phi}}(x)-{{\phi}}(y)}{|x-y|^{1+2s}} \, dy.
\end{split}
\end{equation}
For the sake of simplifying the notation, we will omit the principal value
in~$A(x)$ in the following computations.

We observe that, for every~$\gamma>0$,
\begin{equation}\label{794-5684hdgjsgdbsggklpyoyi}\begin{split}
\int_{\frac{\kappa}{|x|}}^{\frac12}  \theta^{-\gamma}\big(1+O(\theta)\big)\,d\theta
&
=\begin{cases}\displaystyle \frac1{1-\gamma}\left(2^{\gamma-1} -\kappa^{1-\gamma}|x|^{\gamma-1}\right) \quad &{\mbox{if }}\gamma\neq1\\
\displaystyle\ln\left(\frac{|x|}{2\kappa}\right)\quad &{\mbox{if }}\gamma=1
\end{cases}\\&\qquad\qquad
+\begin{cases}
\displaystyle O(1)\quad &{\mbox{if }}\gamma\in(0,2)\\
\displaystyle O(\ln |x|)\quad &{\mbox{if }}\gamma=2\\
\displaystyle O(|x|^{\gamma-2})\quad &{\mbox{if }}\gamma\in(2,+\infty).
\end{cases}
\end{split}\end{equation}

Now we claim that
\begin{equation}\label{x}
\lim_{x \to -\infty} \lvert x \rvert^{2s} A(x) =0.
\end{equation}
For this, we change variable~$y:=|x|\theta$ to see that
\begin{equation}\label{5748932fghdjstryeuryue}
\begin{split}&
|x|^{2s}A(x)=|x|^{2s}
\int_{-\infty}^{-\kappa} \frac{{{\phi}}(y)-{{\phi}}(x)}{|x-y|^{1+2s}} \, dy =
C_1|x|^{2s}
\int_{-\infty}^{-\kappa} \frac{|y|^{-\alpha}-|x|^{-\alpha}}{|x-y|^{1+2s}} \, dy\\&\qquad
=C_1|x|^{-\alpha}
\int_{-\infty}^{-\frac\kappa{|x|}} \frac{|\theta|^{-\alpha}-1}{|1+\theta|^{1+2s}} \, d\theta
=-C_1|x|^{-\alpha}
\int_{-\infty}^{-\frac\kappa{|x|}} f(\theta) \, d\theta.\end{split}
\end{equation}
We point out that
\begin{equation*}\begin{split}&
\int_{-\frac32}^{-\frac12} f(\theta)\, d\theta 
=\int_{-\frac32}^{-\frac12}\frac{(1 -\alpha(1+\theta) +O(|1+\theta|^2)) -1}{\lvert 1 +\theta\rvert^{1+2s}}\, d\theta 
\\&\qquad =-
\int_{-\frac32}^{-\frac12}\frac{\alpha(1+\theta) }{\lvert 1 +\theta\rvert^{1+2s}}\, d\theta +
\int_{-\frac32}^{-\frac12}\frac{O(|1+\theta|^2) }{ \lvert 1 +\theta\rvert^{1+2s}}\, d\theta\\&\qquad= 
\int_{-\frac32}^{-\frac12}\frac{O(| 1+\theta |^2)}{\lvert 1 +\theta\rvert^{1+2s}}\, d\theta ,
\end{split}
\end{equation*} which is finite.

As a result, thanks to the limit in~\eqref{a}, we have that
\begin{equation*}
\int_{-\infty}^{-\frac12} f(\theta) \, d\theta \quad{\mbox{is a finite quantity.}}\end{equation*}
This entails that
\begin{equation}\label{bvcnxk3275327687998jhgf}
\lim_{x\to-\infty}|x|^{-\alpha}\int_{-\infty}^{-\frac12} f(\theta) \, d\theta=0.
\end{equation}

Furthermore, 
\begin{equation*}\begin{split}&
\int_{-\frac12}^{-\frac\kappa{|x|}} f(\theta) \, d\theta
=\int_{-\frac12}^{-\frac\kappa{|x|}} \frac{1-|\theta|^{-\alpha}}{|1+\theta|^{1+2s}} \, d\theta\\&\qquad = \frac1{2s}\left( 2^{2s}-\left(\frac{|x|}{|x|-\kappa}\right)^{2s}\right)
-\int_{-\frac12}^{-\frac\kappa{|x|}}  \frac{d\theta}{
|\theta|^{\alpha}|1+\theta|^{1+2s}} \\&\qquad=
\frac1{2s}\left( 2^{2s}-\left(\frac{|x|}{|x|-\kappa}\right)^{2s}\right)
-\int_{-\frac12}^{-\frac\kappa{|x|}}  \frac{1+O(|\theta|)}{|\theta|^{\alpha}}\,d\theta
\\&\qquad=
\frac1{2s}\left( 2^{2s}-\left(\frac{|x|}{|x|-\kappa}\right)^{2s}\right)
-\int_{\frac\kappa{|x|}}^{\frac12}  \frac{1+O(|\theta|)}{|\theta|^{\alpha}}\,d\theta.
\end{split}\end{equation*}
From this and~\eqref{794-5684hdgjsgdbsggklpyoyi} (used with~$\gamma:=\alpha$),
we conclude that
$$ \lim_{x\to-\infty}|x|^{-\alpha}
\int_{-\frac12}^{-\frac\kappa{|x|}} f(\theta) \, d\theta=0.
$$
Gathering this formula and~\eqref{bvcnxk3275327687998jhgf},
and using them into~\eqref{5748932fghdjstryeuryue}, we obtain~\eqref{x}.

Also, we observe that if~$x<-2\kappa$ and~$|y|\le\kappa$, then~$|x-y| \geq |x|/2$.
Hence,
\begin{equation*}
| B(x)| \leq 2^{1+2s} |x|^{-(1+2s)} \int_{-\kappa}^{\kappa} |{{\phi}}(y)-{{\phi}}(x)| \, dy\leq  2^{3+2s}\kappa \|{{\phi}}\|_{L^\infty(\R)} |x|^{-(1+2s)},
\end{equation*}
leading to
\begin{equation}\label{m}
 \lim_{x \to -\infty} \lvert x \rvert^{2s}  B(x) =0.
\end{equation}

We now claim that
\begin{equation}\label{olu}
\lim_{x\to -\infty} |x|^{2s}D(x)= \frac{1}{s}.
\end{equation}
To this end, we remark that, changing variable~$y:=|x|\theta$,
\begin{equation*}
\begin{split}
D(x) &= - \int_{\kappa}^{+\infty}\frac{C_1|x|^{-\alpha}-2+C_2|y|^{-\beta}}{|x-y|^{1+2s}} \, dy
=  -\int_{\frac{\kappa}{|x|}}^{+\infty}\frac{-2 +C_1|x|^{-\alpha} +C_2|x|^{-\beta} \theta^{-\beta} } {|x|^{2s}(1+\theta)^{1+2s}} \, d\theta\\
&= \lvert x \rvert^{-2s}\left( 2 - C_1\lvert x \rvert^{-\alpha} \right)
\left( \int_{\frac\kappa{|x|}}^{+\infty}\frac{d\theta}{(1+\theta)^{1+2s}}  \right) - C_2\lvert x \rvert^{-(\beta +2s)}\int_{\frac{\kappa}{|x|}}^{+\infty} \frac{\theta^{-\beta}}{(1+\theta)^{1+2s}} \, d\theta\\
&= |x|^{-2s}\left( 2 - C_1\lvert x \rvert^{-\alpha} \right) \frac{1}{2s} \left(1+\frac{\kappa}{|x|} \right)^{-2s}  - C_2 |x|^{-(\beta+2s)} \int_{\frac{\kappa}{|x|}}^{+\infty} \frac{\theta^{-\beta}}{(1+\theta)^{1+2s}} \, d\theta.
\end{split}
\end{equation*}
We observe that
\begin{eqnarray*}&&
\int_{\frac{\kappa}{|x|}}^{+\infty} \frac{\theta^{-\beta}}{(1+\theta)^{1+2s}} \, d\theta=
\int_{\frac{\kappa}{|x|}}^{1/2} \frac{\theta^{-\beta}}{(1+\theta)^{1+2s}} \, d\theta +
\int_{\frac12}^{+\infty} \frac{\theta^{-\beta}}{(1+\theta)^{1+2s}} \, d\theta\\
&&\qquad
=
\int_{\frac{\kappa}{|x|}}^{1/2} 
\theta^{-\beta}\big(1+O(\theta)\big)\, d\theta +
\int_{\frac12}^{+\infty} \frac{\theta^{-\beta}}{(1+\theta)^{1+2s}} \, d\theta.
\end{eqnarray*}
As a result,
\begin{equation*}\begin{split}
|x|^{2s}D(x)=\;&
\left( 2 - C_1\lvert x \rvert^{-\alpha} \right) \frac{1}{2s} \left(1+\frac{\kappa}{|x|} \right)^{-2s}\\&\;
-C_2  |x|^{-\beta}
\int_{\frac{\kappa}{|x|}}^{\frac12} 
\theta^{-\beta}\big(1+O(\theta)\big)\, d\theta -C_2  |x|^{-\beta}
\int_{\frac12}^{+\infty} \frac{\theta^{-\beta}}{(1+\theta)^{1+2s}} \, d\theta.
\end{split}
\end{equation*}
which, coupled with~\eqref{794-5684hdgjsgdbsggklpyoyi} (used here with~$\gamma:=\beta$),
gives~\eqref{olu}.

Now, since~$L_s{{\phi}}(x)= A(x) + B(x) + D(x)$,
the desired limit in~\eqref{yuuy}
follows from~\eqref{x}, \eqref{m} and~\eqref{olu}. \end{proof}

The next result concerns the asymptotic decay of~$L_s$ when applied to 
functions given by derivatives of polynomials.  
The guiding example is the case of~$L_s{\phi}^{(i)}$, which will be 
treated later in Corollary~\ref{4tyyhj}.  
However, due to the novelty of this result, we prefer to present it here 
in its full generality.

\begin{prop}\label{proppo}
Let~$\alpha$, $\beta$, $\kappa$, $C\in (0,+\infty)$ and~$ i \in \N\setminus\{0\}$. Let~$u \in C^{i+1,1}(\R)$ be such that~$u'>0$ and,
for every~$\ell=1,\ldots,i$,
\begin{equation}\label{labellliamo} 
u^{(\ell)}(x) \in \begin{cases}
(0,C |x|^{-\alpha-\ell}) &\mbox{for} \ x\leq -\kappa, \\
( - C |x|^{-\beta-\ell},0 ) &\mbox{for} \ x\geq \kappa \ \mbox{and }\ell \ \mbox{even},\\
(0, C |x|^{-\beta-\ell} ) &\mbox{for} \ x\geq \kappa \ \mbox{and } \ell \ \mbox{odd},
\end{cases}
\end{equation}
and
\begin{equation}\label{i4grf}\begin{split}
&|x|^3 \Vert u^{(i+2)}\Vert_{L^{\infty}\left( \frac{3x}2,\,\frac x 2\right)} =o(|x|^{1-i}) \quad\mbox{as}\quad x\to-\infty
\\{\mbox{and }}\qquad &
x^3 \Vert u^{(i+2)}\Vert_{L^{\infty}\left(\frac x 2,\, \frac{3x}2\right)} =o(x^{1-i}) \quad\mbox{as}\quad x\to+\infty.
\end{split}\end{equation}

Then,
\begin{equation}\label{leb}
\lim_{x \to \pm \infty} |x|^{i+2s} L_s u^{(i)}(x) =  (\mp 1)^{i-1} \,\frac{\Vert u'\Vert_{L^1(\R)}\Gamma(i+2s)}{\Gamma(1+2s) }.
\end{equation}
\end{prop}

\begin{proof}
We establish~\eqref{leb} first in the case~$x \to +\infty$.
For this, let~$x\geq 2\kappa$. 

We claim that\footnote{As usual, when~$i=1$ the summation in~\eqref{rrrrr} is intended to be void.}
\begin{equation}\label{rrrrr}\begin{split}
&x^{1+2s}\int_{-\infty}^{\frac x 2}  \frac{u^{(i)}(y) - u^{(i)}(x)}{|x-y|^{1+2s}}\, dy 
\\&\quad=\int_{-\frac{x}2}^{\frac{x}2} u^{(i)}(y)\,dy+
x^{-j}\sum_{j=1}^{i-1} \binom{-1-2s}{j} (-1)^j\int_{-\infty}^{+\infty} u^{(i)}(y) y^j \, dy \\
&\qquad\quad+O(x^{1-\beta-i})+O(x^{1-\alpha-i})+O(x^{-i})\\&\qquad\quad+\begin{cases}
\displaystyle O(x^{1-\alpha-i})\quad &{\mbox{if }}\alpha\in(0,1) \\
\displaystyle O(x^{-i}\ln x)\quad &{\mbox{if }}\alpha=1\\
\displaystyle O(x^{-i})\quad &{\mbox{if }}\alpha\in(1,+\infty) \end{cases}
\\&\qquad\quad +
\begin{cases}
\displaystyle O(x^{1-\beta-i})\quad &{\mbox{if }}\beta\in(0,1) \\
\displaystyle O(x^{1-\beta-i}\ln x)\quad &{\mbox{if }}\beta=1\\
\displaystyle O(x^{-i})\quad &{\mbox{if }}\beta\in(1,+\infty)  .
\end{cases}
\end{split}\end{equation}

To establish this, we make some preliminary observations.
For all~$j \in \N$ with~$j<i$, we have that~$\alpha+i-j>1$, and therefore, thanks to~\eqref{labellliamo},
\begin{equation}\label{vierotjpu8qe7320975098765}
\left|x^{-j}\int_{-\infty}^{-\frac x 2}u^{(i)}(y) y^j \, dy \right| \leq  Cx^{-j} 
\int_{-\infty}^{-\frac x 2}|y|^{-\alpha-i+j}\, dy = O( x^{1-\alpha-i}).
\end{equation}
Similarly,
\begin{equation}\label{vierotjpu8qe732097509876533}
\left|x^{-j}\int^{+\infty}_{\frac x 2}u^{(i)}(y) y^j \, dy \right| \leq  Cx^{-j} 
\int^{+\infty}_{\frac x 2}|y|^{-\beta-i+j}\, dy = O( x^{1-\beta-i}).
\end{equation}

Moreover, using again~\eqref{labellliamo},
\begin{equation}\label{vierotjpu8qe73209750987652}
\left| x^{-i}\int_{-\frac{x}2}^{- \kappa} u^{(i)}(y) | y|^i\, dy\right| \le C x^{-i}\int_{-\frac{x}2}^{- \kappa}  | y|^{-\alpha}\, dy
=\begin{cases}
\displaystyle O(x^{1-\alpha-i})\quad &{\mbox{if }}\alpha\in(0,1) \\
\displaystyle O(x^{-i}\ln x)\quad &{\mbox{if }}\alpha=1\\
\displaystyle O(x^{-i})\quad &{\mbox{if }}\alpha\in(1,+\infty)
\end{cases}
\end{equation} and analogously
\begin{equation}\label{vierotjpu8qe73209750987652BIS}
\left| x^{-i}\int^{\frac{x}2}_{\kappa} u^{(i)}(y)  y^i\, dy\right| \le C x^{-i}\int^{\frac{x}2}_{ \kappa}   y^{-\beta}\, dy
=\begin{cases}
\displaystyle O(x^{1-\beta-i})\quad &{\mbox{if }}\beta\in(0,1) \\
\displaystyle O(x^{-i}\ln x)\quad &{\mbox{if }}\beta=1\\
\displaystyle O(x^{-i})\quad &{\mbox{if }}\beta\in(1,+\infty).
\end{cases}
\end{equation}

In addition, in light of~\eqref{labellliamo} and changing variable~$y:=x\theta$,
\begin{equation*}
\begin{split}&
\left|\int_{-\infty}^{-\frac x 2} u^{(i)}(y) \left( 1-\frac{y}{x}\right)^{-1-2s} \, dy \right| \leq C  \int_{-\infty}^{-\frac x 2}|y|^{-\alpha-i}
\left( 1-\frac{y}{x}\right)^{-1-2s} \, dy \\&\qquad=
 C x^{1-\alpha-i} \int_{-\infty}^{-\frac12}|\theta|^{-\alpha-i} (1-\theta)^{-1-2s}\, d\theta = O(x^{1-\alpha-i}).
\end{split}
\end{equation*}
Using this and changing variable~$y:=x\theta$, we obtain that
\begin{equation}\label{0uj5hg987643980lkjhg}\begin{split}
&x^{1+2s} \int_{-\infty}^{\frac x 2}  \frac{u^{(i)}(y) - u^{(i)}(x)}{|x-y|^{1+2s}}\, dy
=\int_{-\infty}^{\frac x 2}  \big(u^{(i)}(y) - u^{(i)}(x)\big) \left( 1-\frac{y}{ x}\right)^{-1-2s}\, dy\\ &\qquad=
\int_{-\infty}^{\frac x 2}  u^{(i)}(y) \left( 1-\frac{y}{ x}\right)^{-1-2s}\, dy - u^{(i)}(x)\int_{-\infty}^{\frac x 2}  \left( 1-\frac{y}{ x}\right)^{-1-2s}\, dy\\
&\qquad=\int_{-\frac x 2}^{\frac x 2}  u^{(i)}(y) \left( 1-\frac{y }{x}\right)^{-1-2s}\, dy - u^{(i)}(x)\int_{-\infty}^{\frac x 2}  \left( 1-\frac{y}{ x}\right)^{-1-2s}\, dy +O(x^{1-\alpha-i})\\&\qquad
=x\int_{-\frac 1 2}^{\frac12}  u^{(i)}(x\theta) ( 1-\theta)^{-1-2s}\, d\theta - u^{(i)}(x)x\int_{-\infty}^{\frac 1 2} ( 1-\theta)^{-1-2s}\, d\theta+O(x^{1-\alpha-i})
\\&\qquad
=x\int_{-\frac 1 2}^{\frac12}  u^{(i)}(x\theta) ( 1-\theta)^{-1-2s}\, d\theta 
+O(x^{1-\alpha-i})
+O( x^{1-\beta-i})
.\end{split}\end{equation}

Now, a Taylor expansion of~$(1-\theta)^{-1-2s}$ around the origin to the~$(i-1)^{\rm th}$ order\footnote{We recall that we are using the notation for the generalized binomial coefficients, namely
$$\binom{-1-2s}{j}=\frac{(-1-2s)(-2-2s)\cdots(-2s-j)}{j!}.$$} gives that
\begin{equation*}
(1-\theta)^{-1-2s}=
1 + \sum_{j=1}^{i-1} \binom{-1-2s}{j} (-1)^j  \theta^j + O(|\theta|^{i}),\end{equation*}
from which we deduce that
\begin{eqnarray*}&&
x \int_{-\frac12}^{\frac12}u^{(i)}(x\theta)(1-\theta)^{-1-2s}\, d\theta\\&=&
x \int_{-\frac12}^{\frac12}  u^{(i)}(x\theta) \, d\theta
+  x  \sum_{j=1}^{i-1} \binom{-1-2s}{j} (-1)^j  \int_{-\frac12}^{\frac12}  u^{(i)}(x\theta)   \theta^j\, d\theta 
+ x \int_{-\frac12}^{\frac12} u^{(i)}(x\theta) O(| \theta^{i}|)\, d\theta 
\\&=&
\int_{-\frac{x}2}^{\frac{x}2 }  u^{(i)}(y) \, dy
+  x^{-j}\sum_{j=1}^{i-1} \binom{-1-2s}{j} (-1)^j  \int_{-\frac{x}2}^{\frac{x}2 }  u^{(i)}(y)  y^j\, dy
+ x^{-i}\int_{-\frac{x}2}^{\frac{x}2} u^{(i)}(y) O(| y^{i}|)\, dy 
.\end{eqnarray*}
Thus,
using~\eqref{vierotjpu8qe7320975098765}, \eqref{vierotjpu8qe732097509876533}, \eqref{vierotjpu8qe73209750987652} and~\eqref{vierotjpu8qe73209750987652BIS},
\begin{equation*}\begin{split}
&x \int_{-\frac12}^{\frac12}u^{(i)}(x\theta)(1-\theta)^{-1-2s}\, d\theta \\
=\;& \int_{-\frac{x}2}^{\frac{x}2 }  u^{(i)}(y) \, dy
+ x^{-j}\sum_{j=1}^{i-1} \binom{-1-2s}{j} (-1)^j \int_{-\infty}^{+\infty} u^{(i)}(y) y^j \, dy+
x^{-i}\int_{-\kappa}^{\kappa} u^{(i)}(y) O(| y^{i}|)\, dy
\\&\qquad+O(x^{1-\alpha-i})+ O(x^{1-\beta-i})
+\begin{cases}
\displaystyle O(x^{1-\alpha-i})\quad &{\mbox{if }}\alpha\in(0,1) \\
\displaystyle O(x^{-i}\ln x)\quad &{\mbox{if }}\alpha=1\\
\displaystyle O(x^{-i})\quad &{\mbox{if }}\alpha\in(1,+\infty)  
\end{cases}  \\
&\qquad+\begin{cases}
\displaystyle O(x^{1-\beta-i})\quad &{\mbox{if }}\beta\in(0,1) \\
\displaystyle O(x^{-i}\ln x)\quad &{\mbox{if }}\beta=1\\
\displaystyle O(x^{-i})\quad &{\mbox{if }}\beta\in(1,+\infty).
\end{cases}
\end{split}
\end{equation*}

We also notice that~$u^{(i)}$ is bounded in~$[-\kappa,\kappa]$, hence
$$ x^{-i} \int_{- \kappa }^{\kappa }u^{(i)}(y)O(|y|^i)\, d\theta=O(x^{-i}).
$$ Accordingly,
\begin{equation*}\begin{split}
&x \int_{-\frac12}^{\frac12}u^{(i)}(x\theta)(1-\theta)^{-1-2s}\, d\theta \\
=\;& \int_{-\frac{x}2}^{\frac{x}2 }  u^{(i)}(y) \, dy
+ x^{-j}\sum_{j=1}^{i-1} \binom{-1-2s}{j} (-1)^j \int_{-\infty}^{+\infty} u^{(i)}(y) y^j \, dy
\\&\qquad+O(x^{1-\alpha-i}) +O(x^{1-\beta-i})+O(x^{-i})
+\begin{cases}
\displaystyle O(x^{1-\alpha-i})\quad &{\mbox{if }}\alpha\in(0,1) \\
\displaystyle O(x^{-i}\ln x)\quad &{\mbox{if }}\alpha=1\\
\displaystyle O(x^{-i})\quad &{\mbox{if }}\alpha\in(1,+\infty)  
\end{cases}  \\
&\qquad+\begin{cases}
\displaystyle O(x^{1-\beta-i})\quad &{\mbox{if }}\beta\in(0,1) \\
\displaystyle O(x^{-i}\ln x)\quad &{\mbox{if }}\beta=1\\
\displaystyle O(x^{-i})\quad &{\mbox{if }}\beta\in(1,+\infty).
\end{cases}
\end{split}
\end{equation*}
Plugging this information into~\eqref{0uj5hg987643980lkjhg} we thereby obtain~\eqref{rrrrr}.

Now, we claim that
\begin{equation}\label{eqsd}
x^{1+2s} {\rm PV}_x \int_{\frac x 2}^{\frac{3x}2}  \frac{u^{(i)}(y) - u^{(i)}(x)}{|x-y|^{1+2s}}\, dy = o(x^{1-i}).
\end{equation}
To ease notation, in what follows we will omit the principal value.

We recall that~$u^{(i)}\in C^{1,1}(\R)$ and therefore
we are in the position of applying the Fundamental Theorem of Calculus twice, obtaining that, for any~${\theta \in (1/2,3/2)}$,
\begin{equation*}
\begin{split}
u^{(i)}(x\theta)- u^{(i)}(x)=  \;&\int_{x}^{x\theta} u^{(i+1)}(\tau)\,d\tau \\
=\;& u^{(i+1)}(x)x(\theta-1)+ \int_{x}^{x\theta} \left( u^{(i+1)}(\tau)- u^{(i+1)}(x)\right)\,d\tau\\
=\;&  u^{(i+1)}(x)x(\theta-1) +\int_{x}^{x\theta}\left( \int_x^\tau u^{(i+2)}(\eta)\, d\eta\right)\, d\tau.
\end{split}
\end{equation*}
Combining this with the change of variables~$y:=x\theta$, we obtain that
\begin{eqnarray*}
&&x^{1+2s}\int_{\frac x 2}^{\frac{3x}2}  \frac{u^{(i)}(y) - u^{(i)}(x)}{|x-y|^{1+2s}}\, dy = x
\int_{\frac 1 2}^{\frac{3 }2}  \frac{u^{(i)}(x\theta) - u^{(i)}(x)}{|1-\theta|^{1+2s}}\, d\theta\\
&&\qquad =x
\int_{\frac 1 2}^{\frac{3 }2} \left(u^{(i+1)}(x)x(\theta-1) +\int_{x}^{x\theta}\left( \int_x^\tau u^{(i+2)}(\eta)\, d\eta\right)\, d\tau\right)
\frac{d\theta}{|1-\theta|^{1+2s}}\\&&\qquad=
x\int_{\frac 1 2}^{\frac{3 }2} \left(\int_{x}^{x\theta}\left( \int_x^\tau u^{(i+2)}(\eta)\, d\eta\right)\, d\tau\right)
\frac{d\theta}{|1-\theta|^{1+2s}}.
\end{eqnarray*}
Hence, using~\eqref{i4grf},
\begin{eqnarray*}&&
x^{1+2s}\left|\int_{\frac x 2}^{\frac{3x}2}  \frac{u^{(i)}(y) - u^{(i)}(x)}{|x-y|^{1+2s}}\, dy\right| \leq
x^{3}\Vert u^{(i+2)}\Vert_{L^{\infty}\left(\frac x 2,\, \frac{3x}2\right)}\int_{\frac 1 2}^{\frac{3 }2} |1-\theta|^{1-2s} \, d\theta
= o(x^{1-i}),
\end{eqnarray*}
which proves~\eqref{eqsd}, as desired.

Also, by~\eqref{labellliamo} it holds that
\begin{equation}\label{0i26y}\begin{split}
&x^{1+2s} \left|\int_{\frac{3x}2}^{+\infty}  \frac{u^{(i)}(y) - u^{(i)}(x)}{|x-y|^{1+2s}}\, dy \right| = \left|\int_{\frac{3x}2}^{+\infty}  \big(u^{(i)}(y) - u^{(i)}(x)\big)\left(\frac y x-1\right)^{-1-2s} \, dy	\right| \\&\quad\leq
\int_{\frac{3x}2}^{+\infty}  \big(|u^{(i)}(y)|+| u^{(i)}(x)|\big)\left(\frac y x-1\right)^{-1-2s} \, dy	
\le
2 \widetilde{C} x^{-\beta-i} \int_{\frac{3x}2}^{+\infty} \left(\frac y x-1\right)^{-1-2s} \, dy\\&\quad=
2\widetilde{C} x^{1-\beta-i} \int_{\frac32}^{+\infty}(\theta-1)^{-1-2s}\, dy =
O(x^{1-\beta-i}) ,
\end{split}
\end{equation} for some~$\widetilde{C}>0$.

All in all, using the information coming from~\eqref{rrrrr}, \eqref{eqsd} and~\eqref{0i26y}, we deduce that, as~${x\to+\infty}$,
\begin{equation}\label{48hrf}
\begin{split}
&x^{1+2s}L_s u^{(i)}(x) =x^{1+2s}\int_{-\infty}^{+\infty}\frac{u^{(i)}(y)-u^{(i)}(x)
}{|x-y|^{1+2s}}\,dy
\\&\;=\int_{-\frac{x}2}^{\frac{x}2} u^{(i)}(y)\,dy+
x^{-j}\sum_{j=1}^{i-1} \binom{-1-2s}{j} (-1)^j\int_{-\infty}^{+\infty} u^{(i)}(y) y^j \, dy 
\\&\qquad\;+O(x^{1-\beta-i})+O(x^{1-\alpha-i})+O(x^{-i})+o(x^{1-i})
\\&\qquad\;+\begin{cases}
\displaystyle O(x^{1-\alpha-i})\quad &{\mbox{if }}\alpha\in(0,1) \\
\displaystyle O(x^{-i}\ln x)\quad &{\mbox{if }}\alpha=1\\
\displaystyle O(x^{-i})\quad &{\mbox{if }}\alpha\in(1,+\infty) \end{cases}
\\&\qquad\; +
\begin{cases}
\displaystyle O(x^{1-\beta-i})\quad &{\mbox{if }}\beta\in(0,1) \\
\displaystyle O(x^{1-\beta-i}\ln x)\quad &{\mbox{if }}\beta=1\\
\displaystyle O(x^{-i})\quad &{\mbox{if }}\beta\in(1,+\infty)  .
\end{cases}
\end{split}
\end{equation}

We observe that when~$i=1$ the expression in~\eqref{48hrf} boils down to
\begin{equation*}
\begin{split}
x^{1+2s}L_s u'(x) =\;&
\int_{-\frac{x}2}^{\frac{x}2} u'(y)\,dy+O(x^{-\beta})+O(x^{-\alpha})+O(x^{-1})+o(1)
\\&\qquad\;+\begin{cases}
\displaystyle O(x^{-\alpha})\quad &{\mbox{if }}\alpha\in(0,1) \\
\displaystyle O(x^{-1}\ln x)\quad &{\mbox{if }}\alpha=1\\
\displaystyle O(x^{-1})\quad &{\mbox{if }}\alpha\in(1,+\infty) \end{cases}
\\&\qquad\; +
\begin{cases}
\displaystyle O(x^{-\beta})\quad &{\mbox{if }}\beta\in(0,1) \\
\displaystyle O(x^{-\beta}\ln x)\quad &{\mbox{if }}\beta=1\\
\displaystyle O(x^{-1})\quad &{\mbox{if }}\beta\in(1,+\infty)  ,
\end{cases}
\end{split}
\end{equation*}
from which the limit in~\eqref{leb} plainly follows (recalling that~$u'>0$).

Hence, we now focus on the case~$i\ge2$.
For this, we notice that
\begin{equation*}
\int_{-\frac{x}2}^{\frac{x}2} u^{(i)}(y)\,dy
=u^{(i-1)}\left(\frac{x}2\right)-u^{(i-1)}\left(-\frac{x}2\right),
\end{equation*} and therefore, thanks to~\eqref{labellliamo},
the expression in~\eqref{48hrf} becomes
\begin{equation}\label{fgyueh3r83470tfyweghi3o987654}
\begin{split}
x^{1+2s}L_s u^{(i)}(x) =\;&x^{-j}\sum_{j=1}^{i-1} \binom{-1-2s}{j} (-1)^j\int_{-\infty}^{+\infty} u^{(i)}(y) y^j \, dy 
\\&\qquad\;+O(x^{1-\beta-i})+O(x^{1-\alpha-i})+O(x^{-i})+o(x^{1-i})
\\&\qquad\;+\begin{cases}
\displaystyle O(x^{1-\alpha-i})\quad &{\mbox{if }}\alpha\in(0,1) \\
\displaystyle O(x^{-i}\ln x)\quad &{\mbox{if }}\alpha=1\\
\displaystyle O(x^{-i})\quad &{\mbox{if }}\alpha\in(1,+\infty) \end{cases}
\\&\qquad\; +
\begin{cases}
\displaystyle O(x^{1-\beta-i})\quad &{\mbox{if }}\beta\in(0,1) \\
\displaystyle O(x^{1-\beta-i}\ln x)\quad &{\mbox{if }}\beta=1\\
\displaystyle O(x^{-i})\quad &{\mbox{if }}\beta\in(1,+\infty)  .
\end{cases}
\end{split}
\end{equation}

Also, we claim that, for every~$j\in\N\setminus\{0\}$ and~$m\in\N$ with~$m\le j-1$,
\begin{equation}\label{vbcnxkwqro384ty048}\begin{split}
&\frac1{m!}\int_{-\infty}^{+\infty} u^{(j)}(y)  y^{m}\, dy\\
=\;& \sum_{k=0}^{m-1}  \frac{(-1)^k}{(m-k)!}
\int_{-\infty}^{+\infty} \frac{d}{dy}\left( u^{(j-1-k)}(y)  y^{m-k} \right)\, dy+ (-1)^{m} \int_{-\infty}^{+\infty} u^{(j-m)}(y)\, dy.
\end{split}\end{equation}
We prove this by induction over~$j$. When~$j=1$, the result is obvious.
Thus, we suppose the claim to be true for some~$j-1\ge1$ and we want
to show it for~$j$. This goal is accomplished by using the inductive hypothesis,
according to the following calculation (with the index substitution~$\ell:=k+1$):
\begin{eqnarray*}&&
\frac1{m!}
\int_{-\infty}^{+\infty}  u^{(j)}(y)  y^{m}\,dy\\&&\qquad=
\frac1{m!}
\int_{-\infty}^{+\infty} \frac d{dy}\big( u^{(j-1)}(y)\, y^{m}\big)\,dy
- \frac1{(m-1)!}
\int_{-\infty}^{+\infty} u^{(j-1)}(y)\, y^{m-1}\,dy
\\&&\quad=\frac1{m!}
\int_{-\infty}^{+\infty} \frac d{dy}\big( u^{(j-1)}(y)\, y^{m}\big)\,dy
\\&&\quad\qquad-\left(
\sum_{k=0}^{m-2}  \frac{(-1)^k}{(m-1-k)!}
\int_{-\infty}^{+\infty} \frac d{dy}\big(u^{(j-2-k)}(y)\, y^{m-1-k}\big)\,dy
+(-1)^{m-1}\int_{-\infty}^{+\infty} u^{(j-m)}(y)\,dy
\right)\\&&\quad= \frac1{m!}
\int_{-\infty}^{+\infty} \frac d{dy}\big( u^{(j-1)}(y)\, y^{m}\big)\,dy
\\&&\quad\qquad+
\sum_{\ell=1}^{m-1}  \frac{(-1)^\ell}{(m-\ell)!}
\int_{-\infty}^{+\infty} \frac d{dy}\big(u^{(j-1-\ell)}(y)\, y^{m-\ell}\big)\,dy
+(-1)^{m}\int_{-\infty}^{+\infty} u^{(j-m)}(y)\,dy\\&&\quad=\sum_{\ell=0}^{m-1}  \frac{(-1)^\ell}{(m-\ell)!}
\int_{-\infty}^{+\infty} \frac d{dy}\big(u^{(j-1-\ell)}(y)\, y^{m-\ell}\big)\,dy
+(-1)^{m}\int_{-\infty}^{+\infty} u^{(j-m)}(y)\,dy.
\end{eqnarray*}
The proof of~\eqref{vbcnxkwqro384ty048} is thereby complete.

Now, using~\eqref{labellliamo}, we point out that, for any~$k\in \N$ such that~$k\leq i-2$,
\begin{equation*}
u^{(i-1-k)}(y)y^{i-1-k} = \begin{cases}
\displaystyle O(|y|^{-\alpha})\quad &{\mbox{as }}y\to-\infty,\\ 
\displaystyle O(|y|^{-\beta}) \quad &{\mbox{as }} y\to+\infty. \end{cases}
\end{equation*}
This and the Fundamental Theorem of Calculus imply that
\begin{equation*}
\int_{-\infty}^{+\infty} \frac d {dy} \left(u^{(i-1-k)}(y)y^{i-1-k}\right) \, dy =0.
\end{equation*}
This, together with~\eqref{vbcnxkwqro384ty048} (used here with~$j:=i$ and~$m:=i-1$), gives that
\begin{eqnarray*}&& 
\int_{-\infty}^{+\infty} u^{(i)}(y)\, y^{i-1}\,dy  
\\&&\quad=\sum_{k=0}^{i-2}  \frac{(-1)^k(i-1)!}{(i-1-k)!}
\int_{-\infty}^{+\infty} \frac{d}{dy}\left( u^{(i-1-k)}(y)  y^{i-1-k} \right)\, dy+ (-1)^{i-1}(i-1)! \int_{-\infty}^{+\infty} u'(y)\, dy
\\&&\quad
=(-1)^{i-1}(i-1)!\int_{-\infty}^{+\infty} u'(y)\,dy.\end{eqnarray*}
Thus, since~$u'>0$, we find that
\begin{equation}\label{plkk}
\int_{-\infty}^{+\infty} u^{(i)}(y)\, y^{i-1}\,dy  =(-1)^{i-1}(i-1)!\,\|u'\|_{L^1(\R)}.\end{equation}

Similarly, for any~$k$, $m\in \N$ such that~$k$, $m\leq i-2$,
\begin{equation*}
u^{(i-1-k)}(y)  y^{m-k} = \begin{cases}
\displaystyle O(|y|^{m+1-\alpha-i})\quad &{\mbox{as }}y\to-\infty,\\ 
\displaystyle O(|y|^{m+1-\beta-i}) \quad &{\mbox{as }} y\to+\infty, \end{cases}
\end{equation*}  which,
together with the Fundamental Theorem of Calculus, leads to
\begin{equation*}
\int_{-\infty}^{+\infty}\frac{d}{dy}\left( u^{(i-1-k)}(y)  y^{m-k} \right) \, dy =0.
\end{equation*}
Also, thanks to~\eqref{labellliamo},
$$ \int_{-\infty}^{+\infty} u^{(i-m)}(y)\, dy= \lim_{x\to+\infty}u^{(i-m-1)}(x)-\lim_{x\to-\infty}u^{(i-m-1)}(x)=0.$$
Thus, exploiting~\eqref{vbcnxkwqro384ty048} with~$i:=j$, we conclude that, for all~$m\in\N$ with~$m\le i-2$,
\begin{equation*}\begin{split}
&\int_{-\infty}^{+\infty} u^{(i)}(y)  y^{m}\, dy\\
=\;& \sum_{k=0}^{m-1}  \frac{(-1)^k{m!}}{(m-k)!}
\int_{-\infty}^{+\infty} \frac{d}{dy}\left( u^{(i-1-k)}(y)  y^{m-k} \right)\, dy+ (-1)^{m}{m!} \int_{-\infty}^{+\infty} u^{(i-m)}(y)\, dy\\
=\;&0.
\end{split}\end{equation*}

Hence, plugging this and~\eqref{plkk} into~\eqref{fgyueh3r83470tfyweghi3o987654}, we finally obtain that
\begin{equation*}
\begin{split}
&x^{1+2s}L_s u^{(i)}(x) 
\\&\;=
x^{1-i} \binom{-1-2s}{i-1} (-1)^{i-1}\int_{-\infty}^{+\infty} u^{(i)}(y) y^{i-1} \, dy 
+O(x^{1-\beta-i})+O(x^{1-\alpha-i})+O(x^{-i})+o(x^{1-i})
\\&\;\qquad+\begin{cases}
\displaystyle O(x^{1-\alpha-i})\quad &{\mbox{if }}\alpha\in(0,1) \\
\displaystyle O(x^{-i}\ln x)\quad &{\mbox{if }}\alpha=1\\
\displaystyle O(x^{-i})\quad &{\mbox{if }}\alpha\in(1,+\infty) \end{cases}
\\&\;\qquad +
\begin{cases}
\displaystyle O(x^{1-\beta-i})\quad &{\mbox{if }}\beta\in(0,1) \\
\displaystyle O(x^{1-\beta-i}\ln x)\quad &{\mbox{if }}\beta=1\\
\displaystyle O(x^{-i})\quad &{\mbox{if }}\beta\in(1,+\infty)  
\end{cases}\\
&\;= 
x^{1-i} \binom{-1-2s}{i-1} (i-1)!\,\|u'\|_{L^1(\R)}
+O(x^{1-\beta-i})+O(x^{1-\alpha-i})+O(x^{-i})+o(x^{1-i})
\\&\;\qquad+\begin{cases}
\displaystyle O(x^{1-\alpha-i})\quad &{\mbox{if }}\alpha\in(0,1) \\
\displaystyle O(x^{-i}\ln x)\quad &{\mbox{if }}\alpha=1\\
\displaystyle O(x^{-i})\quad &{\mbox{if }}\alpha\in(1,+\infty) \end{cases}
\\&\;\qquad +
\begin{cases}
\displaystyle O(x^{1-\beta-i})\quad &{\mbox{if }}\beta\in(0,1) \\
\displaystyle O(x^{1-\beta-i}\ln x)\quad &{\mbox{if }}\beta=1\\
\displaystyle O(x^{-i})\quad &{\mbox{if }}\beta\in(1,+\infty).
\end{cases}  
\end{split}
\end{equation*}
As a result,
\begin{equation}\label{few6409379gjjeilswh}
\lim_{x\to+\infty}x^{i+2s}L_s u^{(i)}(x)=\binom{-1-2s}{i-1} (i-1)!\,\|u'\|_{L^1(\R)}.
\end{equation}

We now observe that
\begin{eqnarray*}
&&\binom{-1-2s}{i-1} (i-1)!=(-1-2s)(-2-2s)\cdots(1-i-2s)
\\&&\qquad=(-1)^{i-1}(1+2s)(2+2s)\cdots(i-1+2s)=(-1)^{i-1}
\frac{\Gamma(i+2s)}{\Gamma(1+2s)}.
\end{eqnarray*}
From this and~\eqref{few6409379gjjeilswh} we obtain the desired limit in~\eqref{leb}
as~$x\to+\infty$ for all~$i\geq2$.

In order to deal with the limit as~$x\to-\infty$ in~\eqref{leb}, we define~$v(x):=-u(-x)$.
We notice that~$v^{(i)}(x)=(-1)^{i+1}u^{(i)}(-x)$ for all~$i\ge1$.
Therefore~$v'>0$ and~$v$ satisfies the assumptions in~\eqref{labellliamo} and~\eqref{i4grf} (with the roles of~$\alpha$ and~$\beta$ exchanged).

Moreover,
\begin{eqnarray*}
&&L_sv^{(i)}(x)=\int_{-\infty}^{+\infty}\frac{v^{(i)}(y)-v^{(i)}(x)}{|x-y|^{1+2s}}\,dy
=(-1)^{i+1}\int_{-\infty}^{+\infty}\frac{u^{(i)}(-y)-u^{(i)}(-x)}{|x-y|^{1+2s}}\,dy\\&&\qquad
=(-1)^{i+1}\int_{-\infty}^{+\infty}\frac{u^{(i)}(z)-u^{(i)}(-x)}{|x+z|^{1+2s}}\,dz
=(-1)^{i+1}L_su^{(i)}(-x).
\end{eqnarray*}
As a result,
\begin{eqnarray*}
&&\lim_{x\to-\infty}|x|^{i+2s}L_su^{(i)}(x)
=\lim_{x\to+\infty}|x|^{i+2s}L_su^{(i)}(-x)
=(-1)^{i+1}\lim_{x\to+\infty}|x|^{i+2s}L_sv^{(i)}(x)\\&&\qquad=
(-1)^{i+1}(-1)^{i-1} \,\frac{\Vert v'\Vert_{L^1(\R)}\Gamma(i+2s)}{\Gamma(1+2s) }
=\frac{\Vert u'\Vert_{L^1(\R)}\Gamma(i+2s)}{\Gamma(1+2s) },
\end{eqnarray*}
which completes the proof of~\eqref{leb}.
\end{proof}

A consequence of Propositions~\ref{thone} and~\ref{proppo} is the following result
dealing with the function~${\phi}$ defined in~\eqref{valesem}.

\begin{corol}\label{4tyyhj}
Let~$i \in \N $. Then,
\[
\lim_{x \to\pm  \infty} |x|^{i+2s} L_s {\phi}^{(i)}(x) = 
(\mp 1)^{i-1} \,\frac{2\Gamma(i+2s)}{\Gamma(1+2s) }.
\]
\end{corol}

\begin{proof}
 The case~$i=0$ follows from Proposition~\ref{thone}.

If instead~$i\ge1$, one observes that~$\phi$ satisfies
the regularity assumptions in Proposition~\ref{proppo}. Moreover, $\phi'>0$
and, for every~$\ell=1,\ldots,i$,
\begin{eqnarray*}
\phi^{(\ell)}(x)=
\begin{cases}
C_1\alpha(\alpha+1)\cdots(\alpha +\ell-1)|x|^{-\alpha-\ell} &\mbox{for} \ x\leq -\kappa, \\
C_2(-1)^{\ell+1} \beta(\beta+1)\cdots (\beta+\ell-1)x^{-\beta-\ell} &\mbox{for} \ x\geq \kappa,
\end{cases}
\end{eqnarray*}
which implies that the assumption in~\eqref{labellliamo} is fulfilled. 

Also, if~$x>2\kappa$,
\begin{eqnarray*}
&&x^3\Vert \phi^{(i+2)}\Vert_{L^{\infty}\left( \frac x 2, \frac{3x}2\right)}
=Cx^3 x^{-\beta-i-2} =Cx^{1-\beta-i},
\end{eqnarray*} up to renaming~$C$ at every step.
Therefore,
$$
 \lim_{x\to+\infty} x^{i-1}
x^3\Vert \phi^{(i+2)}\Vert_{L^{\infty}\left(\frac{x}2,\frac{3x}2\right)}
=C\lim_{x\to+\infty} x^{i-1} x^{1-\beta-i} 
=C\lim_{x\to+\infty} x^{-\beta}=0.
$$

A similar argument can be made to check the limit as~$x\to-\infty$, thus showing that
the assumption in~\eqref{i4grf} is also satisfied.

As a result,  recalling that by~\eqref{valesem} we have
\begin{equation*}
\Vert \phi'\Vert_{L^1(\R)} =\int_\R \phi'(x) \, dx= \phi(+\infty)-\phi(-\infty)=2,
\end{equation*} 
when~$i\ge1$ the limits are a consequence of Proposition~\ref{proppo}.
\end{proof}

\begin{subsection}{On the double-well nature of~$V$}\label{xonc}

In this section we gather the ingredients needed to establish~\eqref{48397tasvjkdfbgkewguo}, namely the fact that~$V$ is a double-well potential. While~$V(-1)=0$ by construction (recall~\eqref{pmes}), the value of~$V$ at~$+1$, as well as its behavior in~$(-1,1)$, are not immediate. Besides being of independent interest, the validity of~\eqref{48397tasvjkdfbgkewguo} is also required to prove~\eqref{ham} in Theorem~\ref{minth}.

Now, we point out that in~\cite{CS14} the authors consider layer solutions to
\begin{equation*}
L_s v = f(v) \quad\mbox{in } \R,
\end{equation*}
namely, solutions~$v$ satisfying
\[
v\in(-1,1), \quad v'>0 \quad\mbox{and}\quad \lim_{x\to\pm\infty} v(x)=\pm1.
\]

Relying on the extension formula of the fractional Laplacian, they prove the following result.

\begin{prop}[Theorem~2.2 in~\cite{CS14}]
Let~$s\in(0,1)$ and let~$f\in C^{1,\alpha}(\R)$ for some~$\alpha>\max\{0,1-2s\}$. Let~$G$ be a potential satisfying~$G'=f$.  

If there exists a layer solution~$v$ to~$L_s v=f(v)$, then
\[ G>G(1)=G(-1) \quad\mbox{in}\quad (-1,1). \]
\end{prop}

This statement is, however, too restrictive for our purposes. Indeed, by~\eqref{REgg}, large values of~$\alpha$ and~$\beta$ only ensure Lipschitz regularity for~$V$, which falls short of the assumptions required in~\cite{CS14}.

Nevertheless, we point out in this section that the regularity of~$f$ in~\cite{CS14} is used solely to obtain regularity for the associated boundary layer. Hence, since in our setting the profile~$\phi$ is assumed to be smooth from the start (see~\eqref{valesem}), we can still recover the same characterization of~$V$ as a double-well potential.

We will use the following notation:
\[\R^2_+:=\{(x,y) \in \R^2: \, x\in \R, \, y>0 \} ,\quad \partial\R^2_+:=\{y=0\} \quad\mbox{and}\quad \overline{\R^2_+}:=\R^2_+ \cup \partial\R^2_+.   \]

Let~$s \in (0,1)$ and set 
\begin{equation*}\label{09iu}
d_s :=2^{2s-1} \frac{ \Gamma(s)}{\Gamma(1-2s)} \qquad\mbox{and}\qquad q_s:= s(1-s)\frac{2^{2s}\Gamma\big(\frac{1+2s}2\big)}{\sqrt{\pi}\Gamma(2-s)} ,
\end{equation*}
where~$\Gamma$ denotes the Euler Gamma function. We mention that these quantities coincide respectively with~$d_s$ and~$C_{1,s}$ defined in~\cite[Remark 11]{CS14}.

We define the Poisson kernel in~$\R^2_+$ as
\begin{equation*}
P_s (x,y) :=  \frac{p_s y^{2s}}{\left( |x|^2+|y|^2\right)^{\frac{1+2s}2}},
\end{equation*}
where~$p_s$ is a normalization constant such that~$\int_\R P_s(x,y)\, dx =1$ for all~$y>0$ (see Remark~\ref{ifkt} for an explicit formula for~$p_s$).

We also consider
\begin{equation}\label{hd75}
H_s(x) :=  \frac{p_{s}}{(1+|x|^2)^{\frac{1+2s}{2}}}
\end{equation}
and we observe that
\[P_s(x,y) =  y^{-1}H_s\left(\frac xy\right). \]
As a consequence,  
\begin{equation}\label{bvcnxiwoytr843y8}
 \int_{\R} H_s(z) \, dz  = \int_{\R}  P_s(z,1)  \, dz=1.
\end{equation}

Moreover, given a function~$v:\R\to\R$, we consider in~$\R^2_+$ the convolution in the~$x$ variable
\begin{equation}\label{346ygf}
\bar{u}(\cdot,y):=P_s(\cdot,y) \ast v.
\end{equation}
The change of variable~$\xi:=\tau/ y $ gives
\begin{equation}\label{dyd5}
\begin{split}&
\bar{u}(x,y) = \int_{\R} P_s(\tau,y) v(x-\tau)\, d\tau = \int_{\R} y^{-1}H_s\left(\frac\tau y\right) v(x-\tau)\, d\tau\\&\qquad =
 \int_{\R} H_s(\xi) v(x-y\xi)\, d\xi.
\end{split}
\end{equation}

The following two results characterize~$\bar{u}$ in terms of the extension problem for the fractional Laplacian (see~\cite{MR2354493, CS14} for an introduction to this topic).

\begin{prop}[Remark 3.8 in~\cite{CS14}]\label{h08}
Let~$v \in C(\R)\cap L^{\infty} (\R)$. Then, the function~$\bar{u}$ in~\eqref{346ygf} is the only pointwise solution to
\begin{equation}\label{jky2}
\begin{cases}
{\rm div} (y^{1-2s}\nabla u) =0& \mbox{in} \quad \R^2_+,\\
u = v &\mbox{on} \quad \partial\R^2_+,
\end{cases}
\end{equation}
in the class of functions~$ C(\overline{\R^2_+})\cap L^{\infty}(\overline{\R^2_+})$.

Moreover, let~$v \in C^{2,\beta}_{\rm loc}(\R) \cap L^{\infty}(\R)$ for some~$\beta\in(0,1)$ and~$f \in C(\R)$. 
Then, $v$ solves 
\begin{equation}\label{8gps}
L_s v = f(v) \quad\mbox{in } \R
\end{equation}
if and only if the function~$\bar{u}$ in~\eqref{346ygf} solves
\begin{equation}\label{0ffogf}
\begin{cases}
{\rm div} \left( y^{1-2s} \nabla u\right) = 0 &\mbox{in} \quad \R^2_+,\\
d_s \lim\limits_{y\to 0^+} y^{1-2s} \partial_y u =q_s f(u) &\mbox{on} \quad \partial\R^2_+ ,
\end{cases}
\end{equation} 
and the trace of~$\bar{u}$ on~$\partial \R^2_+$ is~$v$.
\end{prop}

\begin{prop}[Proposition 3.6 in~\cite{CS14}]\label{396f}
Let~$h\in C(\R)$ and~$u \in C^2(\R^2_+)$ be such that ${y^{1-2s}\partial_y u \in C(\overline{\R^2_+}).}$  If~$u$ is a pointwise solution to
\begin{equation*}
\begin{cases}
{\rm div}(y^{1-2s}\nabla u) =0 &\mbox{in} \ \R^2_+,\\
\lim\limits_{y\to 0^+} y^{1-2s}\partial_y u   = h &\mbox{on } \partial\R^2_+,
\end{cases}
\end{equation*}
then~$w:=y^{1-2s}\partial_y u $ is a pointwise solution to
\begin{equation}\label{wddw4}
\begin{cases}
{\rm div} (y^{2s-1}\nabla w)=0&\mbox{in} \ \R^2_+,\\
w=h&\mbox{on } \partial\R^2_+.
\end{cases}
\end{equation}

\end{prop}

The next two propositions provide regularity results for the function~$\bar{u}$ defined in~\eqref{346ygf}, together with some integral estimates. In addition, Proposition~\ref{pi9e} below
gives a self-contained characterization of the quantity~$w:=y^{1-2s}\partial_y\bar{u}$, whose boundedness will play a crucial role in the analysis of the potential associated with the solution~$v$ of~\eqref{8gps} (i.e. the function~$G$ satisfying~$G'=f$). We refer to~\cite[Section 2.3]{MR2354493} for a physical interpretation of~$w$ as a stream function of~$\partial_x\bar{u}$ in~$\R^2_+$.

\begin{prop}\label{pi9e}
Let~$v \in C^{2,\beta}(\R) $ for some~$\beta\in(0,1)$. Then, there exist~$\bar{\beta}\in (0,\min\{\beta,2s\})$ and~$C_1>0$, depending on~$s$ and~$\bar\beta$, such that
\begin{equation}\label{rf5}
\Vert \bar{u} \Vert_{C^{\bar{\beta}}(\R^2_+)} + \Vert \partial_x\bar{u} \Vert_{C^{\bar{\beta}}(\R^2_+)} + \Vert \partial_{xx}\bar{u} \Vert_{C^{\bar{\beta}}(\R^2_+)} \leq C_1 \Vert v \Vert_{C^{2,\bar{\beta}}(\R)}.
\end{equation}

Moreover, it holds that
\begin{equation}\label{67d8}
|\nabla\bar{u}(x,y)|\le\frac{C_2\Vert v \Vert_{L^\infty(\R)}}{y} \quad\mbox{for any~$ y>0$},
\end{equation} 
for some~$C_2>0$ depending on~$s$.

Furthermore, let~$w:= y^{1-2s} \partial_y\bar{u}$. Then,
\begin{equation}\label{AGGBOUNED00}
\Vert w \Vert_{L^{\infty}(\R^2_+)}\le C_3\big(\Vert v\Vert_{L^{\infty}(\R)}+\Vert v''\Vert_{L^{\infty}(\R)}\big),
\end{equation} for some~$C_3>0$ depending on~$s$.

Also, $w$ admits the representation formula\footnote{We point out that the additional constant~$2sp_s$ in~\eqref{sd09} and~\eqref{ceqh}
arises from the fact that our definition of the operator~$L_s$  in~\eqref{fralapc} does not include any renormalizing constant. This choice is motivated by the fact that the dependance on~$s$ is not relevant for the purposes of this paper.}
\begin{equation}\label{sd09}
w(x,y)=2sp_s \big(P_{1-s}(\cdot,y)\ast L_s v\big)(x)=2sp_s  \int_\R H_{1-s}(\xi) L_sv(x-y\xi)\, d\xi
\end{equation}
and satisfies\footnote{Notice that~\eqref{AGGBOUNED00} and~\eqref{ceqh} imply that~$w(x,0)= 2 s p_sL_sv(x)$ is bounded in~$\R$, as expected from the fact that~$v \in C^{2,\beta}(\R)$. The representation formula in~\eqref{sd09} completes the characterization of~$w$ in the upper halfspace.} 
\begin{equation}\label{ceqh}
\lim_{y\to 0^+}w(x,y) = 2 s p_s L_s v (x) \quad\mbox{for any }x \in \R.
\end{equation}
\end{prop}

\begin{proof}
The regularity of~$v$ and~\eqref{dyd5} imply that
\begin{equation*}
\partial_x\bar{u}(x,y) = \int_\R H_s(\xi ) v'(x-y\xi) \, d\xi \qquad\mbox{and}\qquad \partial_{xx}\bar{u}(x,y) = \int_\R H_s(\xi ) v''(x-y\xi) \, d\xi.
\end{equation*}
Hence, recalling~\eqref{bvcnxiwoytr843y8}, we infer that  
\begin{equation}\label{bnvmc43o09i75t9804398ygtijhgro}
\begin{split}&\|\bar u\|_{L^\infty(\R^2_+)}\le
\|v\|_{L^\infty(\R)},\\
&\|\partial_x\bar u\|_{L^\infty(\R^2_+)}\le \|v'\|_{L^\infty(\R)}\\
{\mbox{and }} \qquad &
\|\partial_{xx}\bar u\|_{L^\infty(\R^2_+)}\le\|v''\|_{L^\infty(\R)}.\end{split}\end{equation}

Moreover, let~$\bar{\beta} \in (0,\min\{\beta,\,2s\})$. Then, using again~\eqref{bvcnxiwoytr843y8},
\[ \int_\R  |\xi|^{\bar{\beta}}H_s(\xi)\, d\xi \leq 2 \int_0^1 H_s(\xi)\, d\xi +2p_s\int_1^{+\infty} \xi^{-1-2s+\bar{\beta}}\, d\xi \leq
2+\frac{2p_s}{2s-\bar\beta}.\]
As a consequence, since~$v \in C^{2,\bar{\beta}}(\R)$,
we find that
\begin{equation*}
\begin{split}
|\bar{u}(x_1,y_1)-\bar{u}(x_2,y_2)| &\leq  \int_{\R} H_s(\xi) |v(x_1-y_1\xi)-v(x_2-y_2\xi)|\, d\xi\\& \leq 
C \Vert v\Vert_{C^{\beta}(\R)} \left(|x_1-x_2|^{\bar{\beta}} \int_{\R}H_s(\xi)+|y_1-y_2|^{\bar{\beta}} \int_\R |\xi|^{\bar{\beta}}H_s(\xi)\, d\xi \right)\\ &\leq
C(s) \Vert v\Vert_{C^{\beta}(\R)} \left(|x_1-x_2|^{\bar{\beta}}+|y_1-y_2|^{\bar{\beta}}\right).
\end{split}
\end{equation*}
An analogous argument can be done for~$\partial_{x}\bar{u}$ and~$\partial_{xx}\bar{u}$. Thus, recalling also~\eqref{bnvmc43o09i75t9804398ygtijhgro}
we obtain~\eqref{rf5}.

Now we take care of the estimate in~\eqref{67d8}. To this aim, we notice that
since~$v$ is bounded in~$\R$, recalling the definition of~$H_s$ in~\eqref{hd75}, we have, for any~$(x,y)\in \R^2_+$,
\begin{equation*}\label{0ojh4}
\lim_{\xi \to \pm \infty}  H_s(\xi)  v(x-y\xi)=0=\lim_{\xi \to \pm \infty} \xi H_s(\xi)  v(x-y\xi).
\end{equation*}
Also,~$H_s'$ is integrable in~$\R$ and so is~$\xi H'_s$, since
\begin{equation*}
\begin{split}
\int_{\R} |\xi H'_s(\xi)|\, d\xi= p_s(1+2s) \int_\R \frac{|\xi|^2}{(1+|\xi|^2)^{\frac{3+2s}2}}	\, d\xi  <+\infty.
\end{split}
\end{equation*}

Therefore, we can perform an integration by parts and obtain that
\begin{equation}\label{0oiuy}
\begin{split}&
\partial_x\bar{u}(x,y) = \int_\R H_s(\xi) v'(x-y\xi)\, d\xi = -\frac1y\int_\R H_s(\xi) \frac{d}{d\xi}\big(v(x-y\xi)\big)\, d\xi\\&\qquad= 
 \frac1y\int_\R H'_s(\xi)  v(x-y\xi)\, d\xi .
\end{split}
\end{equation}
and
\begin{equation}\label{65edfg}
\begin{split}&
\partial_y\bar{u}(x,y) = -\int_\R \xi H_s(\xi) v'(x-y\xi)\, d\xi = \frac1y\int_\R \xi H_s(\xi) \frac{d}{d\xi}\big(v(x-y\xi)\big)\, d\xi\\&\qquad= -
 \frac1y\int_\R \left(H_s(\xi)+\xi H'_s(\xi)\right)  v(x-y\xi)\, d\xi.
\end{split}
\end{equation}
Thus, \eqref{67d8} follows from~\eqref{0oiuy} and~\eqref{65edfg}.

We now turn our attention to the
function~$w= y^{1-2s}\partial_y\bar{u}$. We firstly check that
\begin{equation}\label{mpo}
\sup_{x\in\R,\,y\geq1} |y^{1-2s}\partial_y\bar{u}(x,y) | \leq C \Vert v \Vert_{L^{\infty}(\R)},
\end{equation}
for some positive~$C$ depending on~$s$.
Indeed,
by~\eqref{65edfg}, we have that, for all~$x\in\R$ and~$y\geq1$,
\begin{equation*}
\begin{split}
|y^{1-2s}\partial_y\bar{u}(x,y)| &= y^{-2s}\big|\int_\R \left(H_s(\xi)+\xi H'_s(\xi)\right)  v(x-y\xi)\, d\xi\big| \\&\leq
 \Vert v \Vert_{L^{\infty}(\R)}\int_\R \left(|H_s(\xi)|+ |\xi|\,| H'_s(\xi)|\right) \, d\xi ,
\end{split}
\end{equation*}
which yields~\eqref{mpo}.

We now establish the following, more delicate claim:
\begin{equation}\label{mpoe}
\sup_{x\in\R,\,y\in(0,1)} |y^{1-2s}\partial_y\bar{u}(x,y)| \leq C\big(\Vert v\Vert_{L^{\infty}(\R)}
+ \Vert v'' \Vert_{L^{\infty}(\R)}\big),
\end{equation}
for some positive~$C$ depending on~$s$.

To check~\eqref{mpoe}, we firstly observe that when~$|\xi|\geq 1$ 
\begin{equation}\label{kok}
\begin{split}
H_s(\xi)+\xi H_s'(\xi) &= \frac{p_s}{(1+|\xi|^2)^{\frac{1+2s}2}} \left(1-\frac{(1+2s)|\xi|^2}{1+|\xi|^2}\right)
\\ &=
\frac{p_s}{(1+|\xi|^2)^{\frac{1+2s}2}}
\Bigl(1-(1+2s)(1+|\xi|^{-2})^{-1}\Bigr)
\\& =
\frac{p_s}{(1+|\xi|^2)^{\frac{1+2s}2}}
\left(1-(1+2s)\left(1-\frac{|\xi|^{-2}}{1+|\xi|^{-2}}\right)
\right)\\&=
\frac{p_s}{(1+|\xi|^2)^{\frac{1+2s}2}}
\left(-2s+O(|\xi|^{-2})\right)
\\&=
\frac{-2sp_s}{(1+|\xi|^2)^{\frac{1+2s}2}}
+O(|\xi|^{-3-2s}).
\end{split}
\end{equation} 
Hence, using the change of variable~$t:=y\xi$ we get
\begin{equation}
\label{0uyd}\begin{split}&
y^{-2s}\int_{\{|\xi|\geq 1\}}\left(H_s(\xi)+\xi H_s'(\xi)\right) (v(x-y\xi)-v(x))\, d\xi \\&\quad=
y^{-2s}\int_{\{|\xi|\geq 1\}} 
\left(\frac{-2s p_s}{(1+|\xi|^2)^{\frac{1+2s}2}}+O(|\xi|^{-3-2s})\right)
(v(x-y\xi)-v(x)) \,d\xi \\&\quad=
\int_{\{|t|\geq y\}} \left(\frac{-2sp_s  }{ (y^2+|t|^2)^{\frac{1+2s}2}}+
y^{2}O(|t|^{-3-2s}) \right)
(v(x-t)-v(x)) \,dt  .
\end{split}
\end{equation}

For readability, we denote by
$$
M_1(x,y):=\int_{\{|t|\geq y\}} \frac{y^{2}(v(x-t)-v(x))}{|t|^{3+2s}} \,dt. $$
Also, we stress that, by symmetry,
\begin{equation*}
\int_{|t|\in(y,1)} \frac{ v'(x) t}{|t|^{3+2s}}\, dt=0 .
\end{equation*}
Therefore, by a Taylor expansion of~$v$, we obtain that, for any~$y\in(0,1)$,
\begin{equation}\label{bvnciewyut894y3}
\begin{split}
|M_1(x,y)| &\leq y^{2}\,\left|\int_{\{|t|\geq y\}}\frac{v(x-t)-v(x)}{|t|^{3+2s}}\, dt\right| \\&\leq
y^{2}\, \left|\int_{\{|t|\in(y,1)\}}\frac{v(x-t)-v(x)}{|t|^{3+2s}}\, dt\right|
+  2y^2 \Vert v \Vert_{L^{\infty}(\R)}\int_{\{|t|\geq 1\}}\frac{dt}{|t|^{3+2s}}\\ &\leq
y^2 \Vert v'' \Vert_{L^{\infty}(\R)}\int_{\{|t|\in(y,1)\}}\frac{dt}{ |t|^{1+2s}} + 2y^2 \Vert v \Vert_{L^{\infty}(\R)}
\int_{\{|t|\geq 1\}}\frac{dt}{|t|^{3+2s}}\\
&=
\frac{2(y^{2-2s}-y^2)\Vert v'' \Vert_{L^{\infty}(\R)}}{2s} +\frac{ 2y^2 \Vert v \Vert_{L^{\infty}(\R)}}{1+s}.
\end{split}
\end{equation}
As a consequence,~$M_1$ is uniformly bounded in~$\R\times (0,1)$ and~$M_1=O(y^{2-2s})$ as~$y \to 0^+$. 

We now consider the case~$|\xi|\leq 1$, namely~$|t|\leq y$ after setting~$t:=y\xi$. To do this, we exploit the first equality in~\eqref{kok} to observe
\begin{equation}\label{64765ubvwoue7y1234}\begin{split}&
\left|H_s\left(\frac{t}{y}\right)+\frac{t}y H_s'\left(\frac{t}y\right)\right| 
=  \frac{p_sy^{1+2s}}{(y^2+|t|^2)^{\frac{1+2s}{2}}} \left| 1-\frac{(1+2s)|t|^2}{y^2+|t|^2} \right|\\&\qquad=
\frac{p_sy^{1+2s}|y^2-2s|t|^2|}{(y^2+|t|^2)^\frac{3+2s}2}  
\le \frac{3p_s y^{3+2s}}{y^{3+2s}} =3p_s.
\end{split}\end{equation}
Moreover, we know that~$\xi H_s(\xi)+\xi^2 H_s'(\xi)$ is an odd function of~$\xi$, hence
\begin{equation}\label{64765ubvwoue7y12342}
v'(x) y\int_{\{|\xi|\leq 1\}}\xi \left(H_s(\xi)+\xi H_s'(\xi)\right)  \, d\xi = 0 \quad\mbox{for any~$y>0$.}
\end{equation}

Consequently, setting 
$$M_2(x,y):=y^{-2s}\int_{\{|\xi|\leq 1\}}\left(H_s(\xi)+\xi H_s'(\xi)\right) (v(x-y\xi)-v(x))\, d\xi,  $$ 
combining together the information in~\eqref{64765ubvwoue7y1234} and~\eqref{64765ubvwoue7y12342} and using a Taylor expansion, we gather
\begin{equation} \label{36}
\begin{split} 
|M_2(x,y)|&=
y^{-2s}\left|\int_{\{|\xi|\leq 1\}}\left(H_s(\xi)+\xi H_s'(\xi)\right) 
(v(x-y\xi)-v(x)+v'(x) y\xi)\, d\xi \right| \\ &=
y^{-1-2s}\left|\int_{\{|t|\leq y\}}\left(H_s\left(\frac{t}y\right)+\frac{t}y H_s'\left(\frac{t}y\right)\right)
(v(x-t)-v(x)+v'(x)t)\, dt \right| \\ &\leq
C y^{-1-2s} \Vert v'' \Vert_{L^{\infty}(\R)} \int_{\{|t|\leq y\}}|t|^2 \, dt \\&=
C y^{2-2s} \Vert v'' \Vert_{L^{\infty}(\R)} .
\end{split}\end{equation}
Then,~$M_2$ is uniformly bounded in~$\R\times (0,1)$ and~$M_2=O(y^{2-2s})$ as~$y \to 0^+$. 

We point out that by~\eqref{65edfg},~\eqref{0uyd} and the fact that
\begin{equation*}
\int_\R  \left( H_s(\xi)+\xi H'_s(\xi)\right) \, d\xi= \lim_{\xi\to+\infty}\xi H_s(\xi)-\lim_{\xi\to-\infty}\xi H_s(\xi)=  0,
\end{equation*}
the following holds
\begin{equation}\label{mnbvc98u7y6t5r}\begin{split}
y^{1-2s}\partial_y\bar{u}(x,y)&=- y^{-2s} \int_{\R}\left(H_s(\xi)+\xi H_s'(\xi)\right) (v(x-y\xi)-v(x))\, d\xi 
\\&=2sp_s \int_{\{|t|\geq y\}} \frac{v(x-t)-v(x) }{(y^2+|t|^2)^{\frac{1+2s}2}}\,dt+O(M_1(x,y))+M_2(x,y).
\end{split}\end{equation}
Therefore, recalling~\eqref{bvnciewyut894y3} and~\eqref{36} we obtain that,
for any~$x\in\R$ and~$y\in(0,1)$,
\begin{equation} \label{ckpp}\begin{split}
\left|y^{1-2s}\partial_y\bar{u}(x,y)\right|& \le
2sp_s \left|\int_{\{|t|\geq y\}} \frac{v(x-t)-v(x) }{(y^2+|t|^2)^{\frac{1+2s}2}}\,dt\right|
+Cy^{2-2s}\left( \Vert v \Vert_{L^{\infty}(\R)}+
\Vert v'' \Vert_{L^{\infty}(\R)}
\right).
\end{split}\end{equation}

We now observe that
\begin{eqnarray*}
&&\left|\int_{\{|t|\geq y\}} \frac{v(x-t)-v(x) }{(y^2+|t|^2)^{\frac{1+2s}2}}\,dt\right|=
\left|\int_{\{|t|\in(y ,1)\}} \frac{v(x-t)-v(x) }{(y^2+|t|^2)^{\frac{1+2s}2}}\,dt+\int_{\{|t|\geq 1\}} \frac{v(x-t)-v(x) }{(y^2+|t|^2)^{\frac{1+2s}2}}\,dt\right|\\&&\qquad
=\left|\int_{\{|t|\in(y ,1)\}} \frac{v(x-t)-v(x)+v'(x)t }{(y^2+|t|^2)^{\frac{1+2s}2}}\,dt+\int_{\{|t|\geq 1\}} \frac{v(x-t)-v(x) }{(y^2+|t|^2)^{\frac{1+2s}2}}\,dt\right|\\&&\qquad\le\Vert v'' \Vert_{L^{\infty}(\R)}
\int_{\{|t|\in(y ,1)\}} \frac{|t|^2 }{|t|^{1+2s}}\,dt+
2\Vert v \Vert_{L^{\infty}(\R)}\int_{\{|t|\geq 1\}} \frac{dt }{|t|^{{1+2s}}}\\&&\qquad\le
\frac{\Vert v'' \Vert_{L^{\infty}(\R)}}{1-s}
+
\frac{2\Vert v \Vert_{L^{\infty}(\R)}}s.
\end{eqnarray*}
{F}rom this and~\eqref{ckpp} we infer~\eqref{mpoe}.

Combining together~\eqref{mpo} and~\eqref{mpoe}, we obtain the bound in~\eqref{AGGBOUNED00}, as desired.

We now exploit~\eqref{mnbvc98u7y6t5r} and the Dominated Convergence Theorem to see that
\begin{equation*}\lim_{y\to 0^+}w(x,y)=
\lim_{y\to 0^+}y^{1-2s}\partial_y\bar{u}(x,y) = 2sp_s\, {\rm PV} \int_\R \frac{v(x-t)-v(x)}{|t|^{1+2s}}\, dt =  2sp_s L_sv(x),
\end{equation*}
that is~\eqref{ceqh}.

We are left to prove the representation formula for~$w$ in~\eqref{sd09}.
To do this, we recall that, by Proposition~\ref{h08}, the function~$\bar{u}$ solves the problem in~\eqref{jky2}. Also, the bound in~\eqref{AGGBOUNED00} together with the regularity of~$H_s$ and~$v$, implies that~$w\in C(\overline{\R^2_+})\cap L^{\infty}(\overline{\R^2_+})$. 
 As a consequence, we can employ Proposition~\ref{396f} to conclude that~$w$ solves~\eqref{wddw4}.

Furthermore, we point out that~\eqref{jky2} coincides with~\eqref{wddw4} (up to switching~$s$ with~$1-s$), hence the convolution in the~$x$ variable~$P_{1-s}(\cdot,y)\ast L_s v$ is the only solution to~\eqref{wddw4} in the class of functions~$C(\overline{\R^2_+})\cap L^{\infty}(\overline{\R^2_+})$. As a consequence, the fact that~$w \in C(\overline{\R^2_+})\cap L^{\infty}(\overline{\R^2_+})$ solves~\eqref{wddw4} shows~\eqref{sd09} and completes the proof.
\end{proof}

\begin{rem}\label{ifkt}
{\rm The limit in~\eqref{ceqh} can also be deduced from~\eqref{0ffogf} in Proposition~\ref{h08}. Hence, combining together these two limits gives an explicit formula for~$p_s$ as }
\[ p_s = \frac{q_s}{2sd_s}= \frac{(1-s)}{\sqrt{\pi}}\frac{\Gamma\big(\frac{1+2s}2\big)\Gamma (1-2s )}{\Gamma(2-s)\Gamma(s)}= \frac{\Gamma\big(\frac{1+2s}2\big)\Gamma (1-2s )}{\sqrt{\pi}\Gamma(1-s)\Gamma(s)} . \]
\end{rem}

\begin{prop}\label{uvr}
Let~$v \in C^{2,\beta}(\R) $ for some~$\beta\in(0,1)$. 
Then, there exists~$C>0$, depending on~$s$, such that 
\begin{equation}\label{5tf}
\int_0^{+\infty}y^{1-2s} |\nabla \bar{u}(x,y)|^2\, dy <C
\big(\Vert v\Vert_{L^{\infty}(\R)}^2+\Vert v'\Vert_{L^{\infty}(\R)}^2+\Vert v''\Vert_{L^{\infty}(\R)}^2 \big)
.
\end{equation}

In addition,
\begin{equation}\label{8543874jhfdsnbbsa43ewdc65tyhg7uyjk}
{\mbox{the integral in~\eqref{5tf} can be differentiated in~$x\in\R$ under the integral sign.}}\end{equation}

Furthermore, assume that
\begin{equation}\label{okr9}
\lim_{x\to\pm\infty} v(x)= L^{\pm}. %\qquad\mbox{and}\qquad \lim_{x\to \pm \infty}v'(x)=0.
\end{equation}
Then,
\begin{equation}\label{5e4}
\lim_{|x|\to +\infty} \int_0^{+\infty}y^{1-2s} |\nabla \bar{u}(x,y)|^2\, dy =0.
\end{equation}
\end{prop}

\begin{proof}
We check that 
\begin{equation}\label{fi7}
\int_0^{+\infty} y^{1-2s} |\partial_x\bar{u}(x,y)|^2\, dy\le C \big(\Vert v\Vert_{L^{\infty}(\R)}^2+\Vert v'\Vert_{L^{\infty}(\R)}^2\big).
\end{equation}
Indeed, we use~\eqref{rf5} (recall, in particular, \eqref{bnvmc43o09i75t9804398ygtijhgro}) and~\eqref{67d8} to see that
\begin{equation}\label{pvdt}
\begin{split}&
\int_0^{+\infty} y^{1-2s} |\partial_x\bar{u}(x,y)|^2\, dy=
\int_0^{1} y^{1-2s} |\partial_x\bar{u}(x,y)|^2\, dy+\int_1^{+\infty} y^{1-2s} |\partial_x\bar{u}(x,y)|^2\, dy
\\&\qquad\le C \big(\Vert v\Vert_{L^{\infty}(\R)}^2+\Vert v'\Vert_{L^{\infty}(\R)}^2\big)
\int_0^{1} y^{1-2s}\, dy
+ C  \Vert v \Vert^2_{L^{\infty}(\R)} \int_1^{+\infty} y^{-1-2s} \, dy \\&\qquad
\le C \big(\Vert v\Vert_{L^{\infty}(\R)}^2+\Vert v'\Vert_{L^{\infty}(\R)}^2\big) ,
\end{split}
\end{equation}
thus completing the proof of~\eqref{fi7}.

Moreover, we show that
\begin{equation}\label{ojht8876y5t4jythbgrfe}
\int_0^{+\infty} y^{1-2s} |\partial_y\bar{u}(x,y)|^2\, dy\leq C\big(\Vert v\Vert_{L^{\infty}(\R)}^2+\Vert v''\Vert_{L^{\infty}(\R)}^2\big).
\end{equation}
Indeed,
from the bound in~\eqref{AGGBOUNED00} we deduce that
\begin{equation}\label{lkjl}\begin{split}&
y^{1-2s}|\partial_y\bar{u}(x,y)|^2 = y^{2s-1} |y^{1-2s}\partial_y\bar{u}(x,y)|^2 \leq \Vert w\Vert^2_{L^{\infty}(\R^2_+)} y^{2s-1}\\&\qquad
\le  C\big(\Vert v\Vert_{L^{\infty}(\R)}^2+\Vert v''\Vert_{L^{\infty}(\R)}^2\big)y^{2s-1}
.\end{split}
\end{equation}
Thus, using~\eqref{67d8} again,
\begin{equation*}
\begin{split}
\int_0^{+\infty} y^{1-2s} |\partial_y\bar{u}(x,y)|^2\, dt& \leq C
\big(\Vert v\Vert_{L^{\infty}(\R)}^2+\Vert v''\Vert_{L^{\infty}(\R)}^2\big)\int_0^1 y^{2s-1}\,dy+
C\Vert v\Vert_{L^{\infty}(\R)}
\int_1^{+\infty} y^{-1-2s}\, dy\\& \leq C\big(\Vert v\Vert_{L^{\infty}(\R)}^2+\Vert v''\Vert_{L^{\infty}(\R)}^2\big),
\end{split}
\end{equation*} which is~\eqref{ojht8876y5t4jythbgrfe}.

{F}rom~\eqref{fi7} and~\eqref{ojht8876y5t4jythbgrfe} we thus deduce~\eqref{5tf}.

Next, we address the claim in~\eqref{8543874jhfdsnbbsa43ewdc65tyhg7uyjk}. To ease the reading, we use the notation~$f(x,y):=y^{1-2s} |\nabla\bar{u}(x,y)|^2$.

We claim that there exists~$g_1 \in L^1((0,1))$ such that
\begin{equation}\label{ift7}
 \big|\partial_x f(x,y)   \big| \leq g_1(y) \quad\mbox{for any $ x \in \R$ and $y\in(0,1)$.}
\end{equation}
The proof of this claim is rather long and technical, and is therefore postponed at the end of the argument. 

We now check that there exists~$g_2 \in L^1((1,+\infty))$ satisfying
\begin{equation}\label{ifte7}
 \big|\partial_x f(x,y)   \big| \leq g_2(y) \quad\mbox{for any $x \in \R$ and $ y\in(1,+\infty)$.}
\end{equation}
Indeed, using~\eqref{67d8},
\begin{equation}\label{UJUJHG9087654312345qwerasdzxcvb}
\begin{split}&
|\partial_x f(x,y) | =| 2 y^{1-2s} \nabla\bar{u}(x,y) \cdot \nabla\partial_x \bar{u} (x,y)| \leq
 2 y^{1-2s} |\nabla\bar{u}(x,y)|\,| \nabla \partial_x\bar{u}(x,y) | \\ &\qquad\leq
C\Vert v \Vert_{L^\infty(\R)} y^{-2s}| \nabla \partial_x\bar{u} (x,y)| .
\end{split}
\end{equation}
Furthermore, we recall~\eqref{0oiuy} and~\eqref{65edfg} and we see that
\begin{eqnarray*}
|\partial_{xx}\bar{u}(x,y) |= \left|
\frac1y\int_\R H'_s(\xi)  v'(x-y\xi)\, d\xi \right|\le \frac{C\|v'\|_{L^\infty(\R)}}{y}
\end{eqnarray*}
and
\begin{eqnarray*}
|\partial_{xy}\bar{u}(x,y)| = \left| -
 \frac1y\int_\R \left(H_s(\xi)+\xi H'_s(\xi)\right)  v'(x-y\xi)\, d\xi\right| \le \frac{C\|v'\|_{L^\infty(\R)}}{y}
\end{eqnarray*}
{F}rom these estimates and~\eqref{UJUJHG9087654312345qwerasdzxcvb} we deduce~\eqref{ifte7} with
$$ g_2(x,y):=\frac{C\Vert v \Vert_{L^\infty(\R)} \|v'\|_{L^\infty(\R)}}{y^{1+2s}}
.$$

As a consequence of~\eqref{ift7} and~\eqref{ifte7}, we have that the function~$g:=\chi_{(0,1)}g_1+
\chi_{(1,+\infty)}g_2$ belongs to~$ L^1((0,+\infty))$ and~$|\partial_x f(x,y) | \leq g(y)$ for any~$y>0$. 
This entails the claim in~\eqref{8543874jhfdsnbbsa43ewdc65tyhg7uyjk}.

We now establish~\eqref{5e4}. To this end, we point out that the function~$y^{1-2s}|\nabla\bar{u}|^2$ is bounded uniformly in~$x$
by a function in~$L^1((0+\infty))$ (recall the computations in~\eqref{pvdt} and~\eqref{lkjl}), and therefore
the Dominated Convergence Theorem applies and we obtain
\begin{equation}\label{of0}
\lim_{x\to\pm \infty} \int_0^{+\infty}y^{1-2s} |\nabla \bar{u}(x,y)|^2\, dy = \int_0^{+\infty}y^{1-2s}\lim_{x\to\pm \infty} |\nabla \bar{u}(x,y)|^2\, dy .
\end{equation}

Next, we aim to show that
\begin{equation}\label{9y7}
\lim_{x\to\pm \infty}| \nabla\bar{u}(x,y)|= 0 \quad\mbox{for any } y>0.
\end{equation}
To this end, we first observe that, in light of~\eqref{okr9},
$$\lim_{x\to\pm\infty}v'(x)=0.$$
As a result, we recall~\eqref{0oiuy} and we use the Dominated Convergence Theorem to see that
$$\lim_{x\to\pm \infty}| \partial_x\bar{u} (x,y)| =
\lim_{x\to\pm \infty} 
\left|\int_\R H_s(\xi) v'(x-y\xi)\, d\xi \right|=0.
$$

Similarly, but using~\eqref{65edfg} and~\eqref{okr9}, we also find that
\begin{equation*}
\begin{split}
\lim_{x\to\pm \infty}| \partial_y\bar{u} (x,y)| =&\lim_{x\to\pm \infty}
\left| -\frac1y\int_\R \left(H_s(\xi)+\xi H'_s(\xi)\right)  v(x-y\xi)\, d\xi\right|\\&=
\frac1y \left| \int_\R \left(H_s(\xi)+\xi H'_s(\xi)\right)  \lim_{x\to\pm \infty} v(x-y\xi)\, d\xi\right|\\&=
\frac{|L^{\pm}|}{y}\left| \int_\R \left(H_s(\xi)+\xi H'_s(\xi)\right)   d\xi\right|\\&=
\frac{|L^{\pm}|}{y}\left| \lim_{\xi\to+\infty}\xi H_s(\xi)-\lim_{\xi\to-\infty}\xi H_s(\xi)\right|=0.
\end{split}
\end{equation*} 

These considerations show~\eqref{9y7} holds. 

From~\eqref{of0} and~\eqref{9y7} we deduce the desired limit in~\eqref{5e4}.

Hence, to complete the proof of Proposition~\ref{uvr}
it only remains to check the claim in~\eqref{ift7}. For this,
let~$x\in\R$ and define, for any~$y\in(0,1)$,
\begin{align*}
I_{\rm near}(y)&:=  \frac1y \int_{-\frac1y}^{\frac1y} \big( H_s(\xi) + \xi H_s'(\xi) \big)(v'(x-y\xi)-v'(x)) \, d\xi \\{\mbox{and }}\quad
I_{\rm far}(y)&:=  \frac1y \int_{\{|\xi|\geq \frac1y\}} \big( H_s(\xi)  + \xi H_s'(\xi) \big)(v'(x-y\xi)-v'(x)) \, d\xi.
\end{align*}

We observe that, for any $|\xi|\geq 1$,
\[ \big|H_s(\xi)+\xi H_s'(\xi) \big|= \frac{p_s}{(1+|\xi|^2)^{\frac{1+2s}2}}\left|1- \frac{(1+2s)|\xi|^2}{1+|\xi|^2}  \right| \leq \frac{C}{ |\xi|^{1+2s}}.\]
Therefore, by a Taylor expansion,
\begin{equation} \label{belf}\begin{split}
\big| I_{\rm near}(y)\big| &\leq 2 \Vert v''\Vert_{L^{\infty}(\R)} \int_{0}^\frac1y\big| H_s(\xi) +\xi H_s'(\xi) \big| \xi \, d\xi \\&\leq 
 2 \Vert v''\Vert_{L^{\infty}(\R)}\left( \int_{0}^1\big|H_s(\xi) +\xi H_s'(\xi) \big| \xi \, d\xi + C \int_{1}^\frac1y\frac{d\xi}{\xi^{2s}} \right) \\&\leq
C \Vert v''\Vert_{L^{\infty}(\R)} \begin{cases}
1-\ln y &\mbox{if } s=\frac12,\\
1+y^{2s-1} &\mbox{if } s\neq\frac12.
\end{cases}
\end{split}
\end{equation}

Furthermore,
\begin{equation}\label{3586}\begin{split}
\big| I_{\rm far}(y)\big| &\leq \frac{4 \Vert v'\Vert_{L^{\infty}(\R)}}y \int_{\frac1y}^{+\infty} \big|H_s(\xi)  + \xi H_s'(\xi) \big| \, d\xi  \\&\leq
\frac{C \Vert v'\Vert_{L^{\infty}(\R)}}y \int_{\frac1y}^{+\infty}\frac{d\xi}{ |\xi|^{1+2s}}   \\&\leq 
C \Vert v'\Vert_{L^{\infty}(\R)}y^{2s-1}. 
\end{split}\end{equation}

Moreover, we observe that
\begin{equation*}
v'(x) \int_\R \big(H_s (\xi) +\xi H_s'(\xi) \big) \, d\xi = v'(x) \int_\R \frac{d }{d\xi}(\xi H_s(\xi)) \, d\xi=0.
\end{equation*} Consequently,
recalling~\eqref{65edfg},
\begin{equation*}
\begin{split}
|\partial_{yx}\bar{u}(x,y) |&= \frac1y \left|\int_{\R}\big(H_s(\xi)+\xi H_s'(\xi)\big) (v'(x-y\xi) -v'(x)) \, d\xi \right| \leq \big| I_{\rm near}(y)\big| +\big| I_{\rm far}(y)\big|.
\end{split}
\end{equation*}

Combining this with~\eqref{belf} and~\eqref{3586}, we conclude that
\begin{equation*}
|\partial_{yx}\bar{u}(x,t) | \leq C \big(\Vert v'\Vert_{L^{\infty}(\R)}+ \Vert v''\Vert_{L^{\infty}(\R)}
\big)\big(1+ y^{2s-1}-\ln y \big) ,
\end{equation*}
for some positive~$C$ depending on~$s$.

As a consequence, using this estimate together with the uniform bound in~\eqref{AGGBOUNED00}, we find that, for any $y\in (0,1)$,
\begin{equation}\label{pou}
 y^{1-2s} | \partial_y\bar{u}(x,y)  \partial_{yx}\bar{u}(x,y) | \leq C  \big(\Vert v\Vert_{L^{\infty}(\R)}+\Vert v'\Vert_{L^{\infty}(\R)}+ \Vert v''\Vert_{L^{\infty}(\R)}\big)^2\big(1+ y^{2s-1}-\ln y \big) 
.
\end{equation}

Furthermore, the bound on the $x$-derivatives of $\bar{u}$ in~\eqref{rf5} (recall in particular~\eqref{bnvmc43o09i75t9804398ygtijhgro}) yields
\begin{equation}\label{epou}
 y^{1-2s}| \partial_x\bar{u}(x,t) \partial_{xx}\bar{u}(x,t)| \leq
C  \Vert v'\Vert_{L^{\infty}(\R)} \Vert v''\Vert_{L^{\infty}(\R)}  y^{1-2s} 
.\end{equation}

Thus, since
\begin{equation*}
\begin{split}
|\partial_x f(x,y)& |= 2 y^{1-2s} | \nabla\bar{u}(x,y) \cdot \nabla \bar{u}_x (x,y)|\\&=
  2 y^{1-2s} \big| \partial_x\bar{u}(x,y) \partial_{xx}\bar{u}(x,y) +  \partial_y\bar{u}(x,y)  \partial_{yx}\bar{u}(x,y) \big|,
\end{split}
\end{equation*} using~\eqref{pou} and~\eqref{epou}
we conclude the proof of~\eqref{ift7}, and therefore of Proposition~\ref{uvr}.
\end{proof}

Now, we exploit the results obtained on~$\bar{u}$ to retrieve the double-well nature of the potential~$G$ such that~$G'=f$.

Propositions~\ref{p09} and~\ref{poas} below can be regarded as the counterparts of Lemmata~5.2 and~5.3 in~\cite{CS14} within our setting, where~$f$ is no more than continuous and~$v\in C^{2,\beta}(\R)$.

\begin{prop}\label{p09}
Let~$v \in C^{2,\beta}(\R)$ for some~$\beta\in(0,1)$ and~$f \in C (\R)$ be such that
\begin{equation*}\label{ekuli}
L_s v = f(v)\quad\mbox{in} \ \R.
\end{equation*}
Furthermore, assume that
\begin{equation*}\label{osdkr9}
\lim_{x\to\pm\infty} v(x)= L^{\pm}.% \quad\mbox{and}\quad \lim_{x\to \pm \infty}v'(x)=0.
\end{equation*}

Then, the potential~$G$ such that~$G'=f$ satisfies
\begin{equation*}
G(L^+)= G(L^-).
\end{equation*}
\end{prop}

\begin{proof}
The proof follows the same structure as in~\cite[Lemma~5.2]{CS14}, making use of 
\begin{itemize}
\item Proposition~\ref{h08}, which relates equation~\eqref{8gps} to the extension problem~\eqref{0ffogf};
\item Proposition~\ref{uvr}, which provides the estimates for
$$\int_0^{+\infty} y^{1-2s} |\nabla \bar{u}|^2\, dy,$$ in particular~\eqref{5tf}, \eqref{8543874jhfdsnbbsa43ewdc65tyhg7uyjk} and~\eqref{5e4}.\qedhere
\end{itemize}
\end{proof}

\begin{prop}\label{poas}
Let~$v \in C^{2,\beta}(\R)$ for some~$\beta\in(0,1)$ and~$f \in C (\R)$ be such that
\begin{equation*}\label{ekul}
L_s v = f(v)\quad\mbox{in} \ \R.
\end{equation*}
Furthermore, assume that
\begin{equation*}\label{okr9g}
\lim_{x\to\pm\infty} v(x)= L^{\pm} \quad\mbox{and}\quad v' >0 \ \mbox{in } \R.
\end{equation*}

Then, the potential~$G$ such that~$G'=f$ satisfies, for any~$x \in \R$ and~$y\geq0$,
\begin{equation}\label{yuo}
\frac{d_s}{q_s}\int_0^y \frac{t^{1-2s}}{2}\,\big( \bar{u}_x^2(x,t)-\bar{u}_y^2(x,t) \big)\,dt
< G(\bar{u}(x,0)) - G(L^+).
\end{equation}
\end{prop}

\begin{proof}
The proof follows the approach in~\cite[Lemma~5.3]{CS14}. 
Besides the ingredients used in Proposition~\ref{p09}, one also employs:
\begin{itemize}
\item the fact that~$\partial_x\bar{u}>0$, which is an immediate consequence of $v'>0$ and formula~\eqref{0oiuy};
\item the Hopf boundary principle, applied through~\cite[Proposition 4.11]{CS14}.\qedhere
\end{itemize}
\end{proof}

From Propositions~\ref{p09} and~\ref{poas} the following result on~$G$ is established:
\begin{corol}\label{finmi}
Let~$v \in C^{2,\beta}(\R)$ for some~$\beta\in(0,1)$ and~$f \in C (\R)$ be such that
\begin{equation*}\label{ekul}
L_s v = f(v)\quad\mbox{in} \ \R.
\end{equation*}
Furthermore, assume that
\begin{equation*}\label{okr9g}
\lim_{x\to\pm\infty} v(x)= L^{\pm} \quad\mbox{and}\quad v' >0 \ \mbox{in } \R.
\end{equation*}

Then, the potential~$G$ such that~$G'=f$ satisfies for any~$x \in \R$ 
\begin{equation*}
 G(\bar{u}(x,0)) > G(L^+)= G(L^-).
\end{equation*}
\end{corol}
\begin{proof}
The desired results plainly follow from Propositions~\ref{p09} and~\ref{poas} (where~\eqref{yuo} is used here with~$y:=0$).
\end{proof}

\end{subsection}

\section{Proof of Theorem~\ref{minth} } \label{minth_proof}   
 
In this section we provide the proof of Theorem~\ref{minth}. 
We preliminarily offer some auxiliary results on~$V$. Then, the proof of Theorem~\ref{minth} is placed at the end of the section.

We will use the notation introduced in Section~\ref{n0biob}.

\begin{prop}\label{nuova}
It holds that~$V \in C^\infty((-1,1))$ and
\begin{equation}\label{vbcn845634897qqwsdfcrtg6tg}
V(r) > V(\pm1) = 0 \quad\mbox{for any } r \in (-1,1).
\end{equation}

Moreover, it holds that

\begin{align}
&\lim_{r \to -1^+} \frac{ V(r)}{(1+r)^{\frac{2s}\alpha+1}}=\frac{\alpha C_1^{-\frac{2s}{\alpha}}}{(2s+\alpha)s} \qquad\mbox{and}\qquad \lim_{r \to 1^-} \frac{ V(r)}{(1-r)^{\frac{2s}\beta+1}} =\frac{\beta C_2^{-\frac{2s}\beta}}{(2s+\beta)s},\label{kof}\\
&\lim_{r \to -1^+} \frac{ V\rq{}(r)}{(1+r)^{\frac{2s}\alpha }}=\frac{C_1^{-\frac{2s}{\alpha}}}{s} \qquad\mbox{and}\qquad \lim_{r \to 1^-} \frac{ V\rq{}(r)}{(1-r)^{\frac{2s}\beta}} =-\frac{C_2^{-\frac{2s}\beta}}{s}.\label{fd6}
\end{align}
\end{prop}

\begin{proof}
We recall that the regularity of~${\phi}$ implies that~$L_s{{\phi}} \in C^{\infty}(\R)$ (see e.g~\cite[Proposition~2.7]{S07}). Moreover, since~${\phi}'>0$, 
by the Inverse Function Theorem
we gather that~${{\phi}}^{-1}\in C^{\infty}((-1,1))$.
As a consequence, by the definition of~$ h$
in~\eqref{i3456789097654bvcxs} we infer that~$h \in C^{\infty}((-1,1))$, and therefore~$V\in C^{\infty}((-1,1))$.

Furthermore, exploiting~\eqref{l} and Proposition~\ref{thone}, we see that
\begin{equation*}
\begin{split}&
\lim_{r \to  -1^+}\frac{ V^{\rq{}}(r)}{(1+r)^{\frac{2s}\alpha }} = \lim_{r \to  -1^+} \displaystyle\frac{L_s{{\phi}} ({{\phi}}^{-1}(r))}{(1+r)^{\frac{2s}\alpha }} =  \lim_{x \to - \infty} \displaystyle\frac{ L_s{{\phi}}(x)}{(1+{{\phi}}(x))^{\frac{2s}\alpha }} \\
&\qquad = \lim_{x \to - \infty}\frac{L_s{{\phi}}(x)}{(C_1|x|^{-\alpha})^{\frac{2s}\alpha }} 
= C_1^{-\frac{2s}{\alpha}} \lim_{x \to - \infty} |x|^{2s}L_s{{\phi}}(x)= \frac{C_1^{-\frac{2s}{\alpha}}}{s}.
\end{split}
\end{equation*}
Similarly,
\begin{equation*}
\begin{split}&
\lim_{r\to 1^-} \displaystyle\frac{ V^{\rq{}}(r)}{(1-r)^{\frac{2s}\beta}} = \lim_{r\to 1^-} \displaystyle\frac{L_s{{\phi}} ({{\phi}}^{-1}(r))}{(1-r)^{\frac{2s}\beta}} =  \lim_{x \to + \infty} \displaystyle\frac{ L_s{{\phi}}(x)}{( 1- {{\phi}}(x))^{\frac{2s}\beta}}\\&\qquad= \lim_{x \to + \infty} \frac{L_s{{\phi}}(x)}{(C_2|x|^{-\beta})^{ \frac{2s}\beta}} 
= C_2^{-\frac{2s}\beta} \lim_{x \to + \infty}|x|^{2s}L_s{{\phi}}(x) = -  \displaystyle\frac{C_2^{-\frac{2s}\beta}}{s}.
\end{split}
\end{equation*}
These computations establish~\eqref{fd6}.

In particular,~\eqref{fd6} implies that
$$V'(1):=\lim_{r\to 1^-}V'(r)=0\qquad{\mbox{and}}\qquad
V'(-1):=\lim_{r\to -1^+}V'(r)=0,$$ 
for any positive values of~$\alpha$ and~$\beta$. This observation, together with the fact that~$V\in C^{\infty}((-1,1))$, entails that~$V' \in C([-1,1])$.

Also,~$\phi$ is smooth, bounded and satisfies
\[ \lim_{x\to\pm\infty}\phi(x) = \pm 1 \quad\mbox{and}\quad \phi'>0.\]
Hence, we can apply Corollary~\ref{finmi} (used here with~$f:=V'$,~$v:=\phi$, and~$L^{\pm}:=\pm1$). In this way,
recalling also that~$V(-1)=0$ by construction (recall~\eqref{pmes}), we obtain that, for all~$r\in(-1,1)$,
$$ V(r)>V(1)=V(-1)=0,$$
which is~\eqref{vbcn845634897qqwsdfcrtg6tg}.

In particular,~\eqref{vbcn845634897qqwsdfcrtg6tg} ensures that
$$ \lim_{r \to  -1^+}V(r)=0=\lim_{r \to 1^-}V(r),$$
which allows us to employ L\rq{}H\^{o}pital\rq{}s Rule and obtain~\eqref{kof} from the following computations:
\begin{equation*}
\lim_{r \to  -1^+}\frac{ V(r)}{(1+r)^{\frac{2s}\alpha+1}}=\lim_{r \to  -1^+}\frac{ \alpha V\rq{}(r)}{(2s+\alpha)(1+r)^{\frac{2s}\alpha}}=\frac{\alpha C_1^{-\frac{2s}{\alpha}}}{(2s+\alpha)s}
\end{equation*}
and
\begin{equation*}
\lim_{r \to 1^-}\frac{ V(r)}{(1-r)^{\frac{2s}\beta+1}}=-\lim_{r \to 1^-}\frac{ \beta V\rq{}(r)}{(2s+\beta)(1-r)^{\frac{2s}\beta}}=\frac{\beta C_2^{-\frac{2s}\beta}}{(2s+\beta)s}.\qedhere
\end{equation*}
\end{proof}

\begin{prop}\label{lemm}
Let~$i \in \N$. Then,
\begin{equation}\label{0ijuhe}
\lim_{r\to -1^+} \frac{V^{(i+1)}(r)}{(1+r)^{\frac{2s}\alpha- i}} \in (-\infty,+\infty) \qquad{\mbox{and}}\qquad \lim_{r\to 1^-} \frac{V^{(i+1)}(r)}{(1-r)^{\frac{2s}\beta- i}} \in (-\infty,+\infty).\end{equation}
\end{prop}

\begin{proof}
We will focus on the proof of the second limit in~\eqref{0ijuhe}, since the argument for first one is analogous.

For the sake of readability, we denote the~$j^{\rm th}$ derivative of~${\phi}$ (with the implicit convention that~${\phi}^{(0)}={\phi}$) for~$x\geq 2\kappa$ as
\begin{equation*}
{\phi}^{(j)}(x) = C_{j} x^{-\beta -j},
\end{equation*} for some~$C_j\in\R$.

We argue by induction and observe that when~$i=0$ the desired result follows from
Proposition~\ref{nuova}. Hence, we now take~$i\geq 1$  and assume that, for every~$k=0,\ldots,i-1$,
\begin{equation}\label{3d0po}
\lim_{r\to 1^-} \frac{V^{(k+1)}(r)}{(1-r)^{\frac{2s}\beta- k}} \in (-\infty,+\infty).
\end{equation}

By the Fa\`a di Bruno Formula,
\begin{equation}\label{obbesux}
\frac{d^{i}}{dx^{i}}  V'({{\phi}}(x)) = \sum \frac{i!}{m_1!\cdots m_{i}!}  V^{(1 + m_1 + \cdots + m_{i})}({{\phi}}(x)) \prod_{j=1}^{i} \left( \frac{{{\phi}}^{(j)}(x)}{j!}\right)^{m_j},
\end{equation}
where the sum runs over all~$i$-tuples of nonnegative integers~$(m_1, \ldots, m_{i})$ satisfying
\begin{equation}\label{11af}
\sum_{j=1}^{i}  m_j j= i.
\end{equation}

We now check that every such~$i$-tuple with~$m_\ell>0$ for some~$\ell \neq 1$ satisfies
\begin{equation} \label{pi0i}
\lim_{x \to +\infty} V^{(1 +m_1+\cdots+m_i)}({{\phi}}(x)) \prod_{j=1}^{i} \left( \frac{{{\phi}}^{(j)}(x)}{j!}\right)^{m_j} x^{i+2s} \in (-\infty,+\infty).
\end{equation}
Indeed, for such~$i$-tuples we have that
\begin{equation*}
1+\sum_{j=1}^{i}m_j \le  m_\ell +\sum_{j=1}^i m_j
=2m_\ell +\sum_{{j=1,\dots,i}\atop{j\neq \ell}} m_j 
\le \ell m_\ell +\sum_{{j=1,\dots,i}\atop{j\neq \ell}} m_j j=\sum_{j=1}^{i}  m_j j
= i.
\end{equation*}
This implies that~$m_1+\cdots+m_i\le i-1$, and therefore
we can exploit the inductive assumption in~\eqref{3d0po} (with~$k:=m_1+\cdots+m_i$)
and obtain that
\begin{equation}\label{jb3} \lim_{r\to 1^- } \frac{V^{(1 + m_1 + \cdots + m_{i})}(r)}{
(1-r)^{\frac{2s}{\beta}-( m_1+\cdots+m_i) }} %= \lim_{x\to+\infty}\frac{V^{(1 + m_1 + \cdots + m_{i})}({\phi}(x))}{x^{\sum_j m_j\beta -2s}} C_{i}^{\sum_j m_j -\frac{2s}{\beta}} 
\in (-\infty,+\infty). \end{equation}

Moreover the condition in~\eqref{11af} gives that
\begin{equation}\label{vr024325Er1widk-13}\begin{split}&
\prod_{j=1}^{i} \left( \frac{{\phi}^{(j)}(x)}{j!} \right)^{m_j}
= \prod_{j=1}^{i} \left( \frac{C_j x^{-\beta-j}}{j!} \right)^{m_j}
= x^{-( m_1(\beta+1)+\cdots +m_i(\beta+i))} \prod_{j=1}^{i} \left( \frac{C_j}{j!} \right)^{m_j} \\
&\qquad\qquad= x^{-i-\beta(m_1+\cdots+ m_i)}\prod_{j=1}^{i} \left( \frac{C_j}{j!} \right)^{m_j}.
\end{split}\end{equation}
This and~\eqref{jb3} lead to
\begin{align*} &\lim_{x \to +\infty} V^{(1 +m_1+\cdots+m_i)}({{\phi}}(x)) \prod_{j=1}^{i} \left( \frac{{{\phi}}^{(j)}(x)}{j!}\right)^{m_j} x^{i+2s} \\&\qquad = \prod_{j=1}^{i} \left( \frac{C_j }{j!} \right)^{m_j}\lim_{x \to +\infty} \frac{V^{(1 +  m_1+\cdots+m_i)}({\phi}(x))}{x^{\beta(m_1+\cdots+m_i) -2s}} \\
&\qquad= C_2^{\frac1\beta}
\prod_{j=1}^{i} \left( \frac{C_j }{j!} \right)^{m_j}\lim_{r \to 1^-} \frac{V^{(1 +  m_1+\cdots+m_i)}(r)}{(1-r)^{\frac{2s}{\beta}-( m_1+\cdots+m_i) }}
\in (-\infty,+\infty), \end{align*}
which proves~\eqref{pi0i}.

Now, combining~\eqref{pi0i} with~\eqref{obbesux}, we obtain that
\begin{equation}\label{cfgcf4}\begin{split}
\lim_{x \to +\infty}x^{i+2s} \frac{d^{i}}{dx^{i}}  V'({{\phi}}(x))
&= C +\lim_{x \to +\infty} V^{(i+1)}({{\phi}}(x)) ({{\phi}}'(x))^{i}x^{i+2s} \\
&= C +\lim_{x \to +\infty} V^{(i+1)}({{\phi}}(x)) (C_{2}\beta)^{i} x^{-i(\beta+1)} x^{i+2s}\\
&= C +\lim_{x \to +\infty} V^{(i+1)}({{\phi}}(x)) (C_{2}\beta)^{i} x^{2s-i\beta}\\
&= C +\lim_{x \to +\infty} V^{(i+1)}({{\phi}}(x)) (C_{2}\beta)^{i} C_2^{\frac{2s}{\beta}-i} (1-{\phi}(x))^{i -\frac{2s}{\beta}},
\end{split}\end{equation} for some~$C\in\R$.

In addition, Corollary~\ref{4tyyhj} yields
\begin{equation}\label{65rf}
\lim_{x\to +\infty} x^{i+2s}L_s {\phi}^{(i)}(x) \in (-\infty,+\infty).
\end{equation}

Now, differentiating~$i$ times the equation in~\eqref{l}
gives that
\begin{equation*}%\label{987uy}
\frac{d^{i}}{dx^{i}}  V'({{\phi}}(x)) =(L_s\phi)^{(i)}(x)=
L_s{\phi}^{(i)}(x),
\end{equation*}
where in the last step we have used~\cite[Proposition~2.1]{OURREC}.

This and~\eqref{65rf} entail that
\begin{equation*}
\lim_{x\to +\infty} x^{i+2s}\frac{d^{i}}{dx^{i}}  V'({{\phi}}(x)) 
=\lim_{x\to +\infty} x^{i+2s}L_s {\phi}^{(i)}(x) \in (-\infty,+\infty).
\end{equation*}
Using this together with~\eqref{cfgcf4}, we conclude that
$$ \lim_{x \to +\infty} V^{(i+1)}({{\phi}}(x)) (C_{2}\beta)^{i} C_2^{\frac{2s}{\beta}-i} (1-{\phi}(x))^{i -\frac{2s}{\beta}}\in (-\infty,+\infty).$$
Therefore
\begin{equation*}
\lim_{r\to 1^-} \frac{V^{(i+1)}(r)}{(1-r)^{\frac{2s}\beta- i}} \in (-\infty,+\infty),
\end{equation*}
which completes the induction step.
\end{proof}

We are now in the position to provide the proof of Theorem~\ref{minth}.

\begin{proof}[Proof of Theorem~\ref{minth}]
The statements in~\eqref{l4b967} and~\eqref{48397tasvjkdfbgkewguo}
follow from Proposition~\ref{nuova}.

Proposition~\ref{nuova} also provides the limits in~\eqref{ham} and~\eqref{hamm} and Proposition~\ref{lemm}
the claim in~\eqref{0ijuhed}.

Next, we take care of the limits in~\eqref{trep1} and~\eqref{trep2}. 
To this aim, we take~$\bar{i}\in \N$ to be the largest integer such that~$2s \geq \bar{i}\beta$. In this way, also recalling Proposition~\ref{lemm} and~\eqref{hamm}, we have that,
for all~$k \in \{0,\cdots,\bar{i}-1\}$,
\begin{equation*}
 \lim_{r \to 1^-} (1-r)^{\frac{2s}\beta-k}=0 \qquad\mbox{and}\qquad \lim_{r\to 1^-} \frac{V^{(k+1)}(r)}{(1-r)^{\frac{2s}\beta-k}} \in (-\infty,+\infty).
\end{equation*}
As a consequence, for all~$k \in \{0,\cdots,\bar{i}-1\}$,
$$\lim_{r\to 1^-}V^{(k+1)}(r)=0.$$

Accordingly, this allows us to apply L\rq{}H\^{o}pital\rq{}s Rule 
and obtain that, for all~$k \in \{1,\cdots,\bar{i}\}$,
\begin{equation*}
\begin{split}
\lim_{r \to 1^-} \frac{V'(r)}{(1-r)^{\frac{2s}\beta}} =  \lim_{r \to 1^-} \frac{V^{(k+1)}(r)}{ \frac{d^k}{dr^k} \left( 1-r\right)^{\frac{2s}\beta}} &= (-1)^k \left(\prod_{j=0}^{k-1} \left( \frac{2s}\beta -j \right)\right)^{-1}\lim_{r \to 1^-} \frac{V^{(k+1)}(r)}{  (1-r)^{\frac{2s}{\beta}-k}} .
\end{split}
\end{equation*}
This and Proposition~\ref{nuova} entail that, for any~$k \in\{1, \ldots, \bar{i}\}$,
$$
\lim_{r \to 1^-} \frac{V^{(k+1)}(r)}{  (1-r)^{\frac{2s}{\beta}-k}}=(-1)^{k+1}
\frac { C_2^{-\frac{2s}\beta}}{s}  \prod_{j=0}^{k-1}\left(  \frac{2s}\beta -j \right) .$$
This proves~\eqref{trep2}, and~\eqref{trep1} is shown analogously.

Now we prove that if~$2s\notin\beta\mathbb N$, then, for every
$k\in\mathbb N$,
\begin{equation}\label{NE8YGF9}
\lim_{r\to1^-}
\frac{V^{(k+1)}(r)}
{(1-r)^{\frac{2s}{\beta}-k}}
=
(-1)^{k+1}
\frac{C_2^{-\frac{2s}{\beta}}}{s}
\prod_{j=0}^{k-1}
\left(\frac{2s}{\beta}-j\right).
\end{equation}

Indeed, let $\bar{i}\in\mathbb N\setminus\{0\}$ be the unique integer such that
\[
(\bar{i}-1)\beta<2s<\bar{i}\beta.
\]
Reasoning as above, we obtain
\[
\lim_{r\to1^-}
\frac{V^{(\bar{i}+1)}(r)}
{(1-r)^{\frac{2s}{\beta}-\bar{i}}}
=
(-1)^{\bar{i}+1}
\frac{C_2^{-\frac{2s}{\beta}}}{s}
\prod_{j=0}^{\bar{i}-1}
\left(\frac{2s}{\beta}-j\right)
\neq0.
\]
Since $\frac{2s}{\beta}-\bar{i}<0$, we have
\[
\lim_{r\to1^-}
(1-r)^{\frac{2s}{\beta}-\bar{i}}
=+\infty
\qquad\text{and}\qquad
\lim_{r\to1^-}|V^{(\bar{i}+1)}(r)|=+\infty.
\]
Thus, L'Hôpital's Rule can be applied again to the ratio in
\eqref{NE8YGF9} with $k=\bar{i}+1$, thus obtaining
\[
\lim_{r\to1^-}
\frac{V^{(\bar{i}+2)}(r)}
{(1-r)^{\frac{2s}{\beta}-\bar{i}-1}}
=
(-1)^{\bar{i}+2}
\frac{C_2^{-\frac{2s}{\beta}}}{s}
\prod_{j=0}^{\bar{i}}
\left(\frac{2s}{\beta}-j\right)
\neq0
\]
and, in turn, that
\[
\lim_{r\to1^-}
(1-r)^{\frac{2s}{\beta}-\bar{i}-1}
=+\infty
\qquad\text{and}\qquad
\lim_{r\to1^-}|V^{(\bar{i}+2)}(r)|=+\infty.
\]
Repeating the same argument
inductively, we obtain \eqref{NE8YGF9} for every $k\in\mathbb N$. The analogous statement holds as $r\to-1^+$, with the same proof.

We now turn to the proof of~\eqref{REgg}. For this, we write
\[ \frac{2s}\alpha= m + \sigma, \quad\mbox{where}\quad m:=\left\lfloor\frac{2s}\alpha\right\rfloor \quad\mbox{and}\quad\sigma\in [0,1). \]
Then, applying~\eqref{trep1} with $i:=m$, we find that
\[ \lim_{r\to-1^+} \frac{V^{(m+1)} (r)}{(1+r)^\sigma} = \frac{1}{s} C_1^{-\frac{2s}\alpha} \prod_{j=0}^{m-1} \left( \frac{2s}\alpha -j\right)<+\infty.\]
Since $V$ is smooth in $(-1,1)$, this limit ensures that $V^{(m+1)}$ extends continuously and remains bounded on~$[-1,0]$. Consequently, $V^{(m)}$ is Lipschitz continuous on~$[-1,0]$, and therefore $V \in C^{\lfloor \frac{2s}\alpha \rfloor,1}([-1,0])$. An analogous argument yields $V \in C^{\lfloor \frac{2s}\beta \rfloor,1}([0,1])$ and completes the proof of~\eqref{REgg}.
\end{proof}

\section{On the transition layer~$\arctan(x)$}\label{atanx} 

Here, we focus on the specific case~$s = 1/2$ and construct a potential~$V$ 
satisfying the Allen–Cahn equation as in~\eqref{l}, where the transition layer is chosen 
to be of the form of~$\arctan(x)$.

\begin{prop}\label{proof}
Let
\begin{equation}\label{newla}
u(x): = \frac{2}{\pi}\arctan(x) \quad\mbox{\rm for any}\quad x \in \R
\end{equation}
and
\[ V_u(\rho) := \frac1{\pi^2}\left( \cos(\pi\rho)-\cos(\pi) \right) \quad\mbox{\rm for any}\quad \rho \in [-1,1].\]

Then, 
\begin{equation}\label{ferqe}
L_{1/2}u(x) =V_u'(u(x))  \quad\mbox{\rm for any}\quad x \in \R. 
\end{equation}
\end{prop}

\begin{proof}
It holds that (see~\cite[Appendix~{L}]{AV19} for a proof of this fact)
\[ L_{1/2}u(x) = - \frac1{\pi}\sin(\pi u(x)) \quad\mbox{\rm for any}\quad x \in \R.\]

Moreover, since
\[  V_u'(\rho) = - \frac{1}{\pi}\sin(\pi \rho),\]
we obtain~\eqref{ferqe}.
\end{proof}

Now, let us consider~${\phi}$ as in~\eqref{valesem} with the specifications~$\alpha=\beta=1$ and~$C_1=C_2=1$, together with the associated potential~$V$ in~\eqref{pmes}, 
in the case~$s=1/2$.

We point out that the transition layer~$u$ defined in~\eqref{newla} and~${\phi}$ 
behave similarly for sufficiently large~$|x|$. Indeed,
\[
\lim_{x\to+\infty} x \left( 1- \frac{2}{\pi}\arctan(x) \right) = \frac{2}{\pi},
\qquad
\lim_{x\to-\infty} x \left( 1+ \frac{2}{\pi}\arctan(x) \right) = - \frac{2}{\pi},
\]
and
\[
u'(x) = \frac{2}{\pi(1+x^2)},
\qquad
{\phi}'(x) =\frac1{ x^{2}}.
\]

Moreover, it holds that 
\[
V_u(\pm1)=V_u'(\pm1)=V(\pm1)=V'(\pm1)=0.
\]

Nevertheless, as noted in Remark~\ref{pollohay}, such similarities between 
the profiles of~$u$ and~${\phi}$ do not carry over to similar potentials. 
In fact, $V_u$ is a cosinus function and all its derivatives are bounded in~$[1,1]$, 
whereas we only have that~$V \in C^{1,1}((-1,1))$.

\section{On Remark~\ref{remmlml}}\label{mlml}

Here we expand on Remark~\ref{remmlml}.  
For the sake of clarity, we shift the analysis at the origin, rather than at~$\pm1$.

\begin{prop}
Let~$\alpha>0$. Then, there exists~$f \in C^{\infty} \in (0,1)$  satisfying

\begin{equation}\label{099p99}
\lim_{x\to 0^+} \frac{f(x)}{x^{\alpha}} = C, 
\end{equation}
for some constant~$C\neq0$, and such that~$f$ is not of class~$C^{0,\beta}(0,\epsilon)$ for any~$\beta\in(0,1]$ and any~$\epsilon>0$.
\end{prop}

\begin{proof}
We construct a function~$f$ in which strong oscillations prevent H\"older regularity near the origin. The techincal details are as follows.

Let~$\beta\in (0,1)$ and set
\begin{equation}\label{0g5t}
\gamma := \frac{2|\beta-\alpha|}{\beta}.
\end{equation}
For large~$n\in \N$, define the points
\[
p_n:= \left(\frac{2}{\pi(2+2n)}\right)^{\frac1{\gamma}}
\qquad {\mbox{and}}\qquad
q_n:= \left(\frac{2}{\pi(1+2n)}\right)^{\frac1{\gamma}}.
\]

A Taylor expansion yields, for large~$n$ and for some~$C_1>0$, that
\begin{equation}\label{0ytfj}
\begin{split}&
|q_n^\alpha-p_n^\alpha |=\left| q_n^{\alpha}\left( 1- \left(\frac{p_n}{q_n}\right)^\alpha \right)\right| =\left| q_n^\alpha \left( 1-\left(\frac{1+2n}{2+2n}\right)^{\frac{\alpha}{\gamma}}\right)\right|=\left|
q_n^\alpha \left( 1-\left(1-\frac1{2+2n}\right)^{\frac\alpha\gamma}\right)\right|\\&\qquad=\left|
q_n^\alpha\left(\frac\alpha{\gamma(2+2n)}+O\left(\frac1{n^2}\right)\right)\right|\leq
\frac{C_1 q_n^\alpha}n.
\end{split}
\end{equation}

Moreover,
\begin{equation*}
\sin (p_n^{-\gamma})= \sin \bigg( \frac{\pi(2+2n)}{2}\bigg) =  \sin ( \pi +n\pi) =0
\end{equation*}
and
\begin{equation*}
|\sin (q_n^{-\gamma})|=\left| \sin \left( \frac{\pi(1+2n)}{2}\right) \right|= \left| \sin \left( \frac{\pi}{2}+n\pi\right) \right| =1.
\end{equation*}

In addition, we stress that, for large~$n$ and for some~$C_2>0$,
\begin{equation}\label{s465d2df}
\left| \frac{ \sin(q_n^{-\gamma})}{\ln q_n }\right|
= \frac{\gamma }{\ln\!\left(\tfrac{\pi(1+2n)}{2}\right)} 
\geq  \frac{C_2  }{\ln n}.
\end{equation}

We now define
\[
f(x):=x^{\alpha}\left(\frac{\sin (x^{-\gamma})}{\ln x} +1\right).
\]
Combining~\eqref{0ytfj} and~\eqref{s465d2df}, we obtain that, for large~$n$ and for some~$C>0$,
\begin{equation}\label{08hkjj}
\begin{split}&
|f(q_n)-f(p_n)|
= \left|q_n^{\alpha} \frac{\sin (q_n^{-\gamma})}{\ln q_n} + q_n^{\alpha} - p_n^{\alpha}\right|
\geq \left|q_n^{\alpha} \frac{\sin (q_n^{-\gamma})}{\ln q_n}\right|-| q_n^{\alpha} - p_n^{\alpha}| \\ 
&\qquad\geq q_n^{\alpha} \left( \frac{C_2}{\ln n}-\frac{C_1}{n} \right)
\geq \frac{C q_n^{\alpha}}{\ln n}.
\end{split}
\end{equation}

Furthermore, from~\eqref{0ytfj} with~$\alpha=1$ we deduce that
\begin{equation*}
|q_n-p_n|\leq \frac{C_1 q_n}{n}.
\end{equation*}
Therefore, using this, \eqref{08hkjj}, the definition of~$q_n$, and the choice of~$\gamma$ in~\eqref{0g5t}, we obtain (up to renaming~$C$) that
\begin{equation*}
\frac{|f(q_n)-f(p_n)|}{|q_n-p_n|^{\beta}} 
\geq  \frac{C q_n^{\alpha}n^{\beta}}{q_n^{\beta}\ln n } 
=  \frac{C n^{\frac{\beta-\alpha}{\gamma}+\beta}}{\ln n} 
=  \frac{C n^{\beta\left(1+\frac{\beta-\alpha}{2|\beta-\alpha|}\right)}}{\ln n} 
\geq \frac{C n^{\frac\beta{2}}}{\ln n}.
\end{equation*}
The right-hand side diverges as~$n\to+\infty$. Since~$p_n$, $q_n\to 0$ as~$n\to+\infty$, we conclude that~$f$ fails to belong to~$C^{0,\beta}$ for any~$\beta\in(0,1]$ in~$(0,\epsilon)$, for any~$\epsilon>0$. 

Nevertheless,~$f \in C^{\infty}(0,1)$ and condition~\eqref{099p99} holds, since
\begin{equation*}
\lim_{x\to0^+} \frac{ f(x)}{x^{\alpha}} = \lim_{x\to0^+} \frac{\sin (x^{-\gamma})}{\ln x} +1 = 1.\qedhere
\end{equation*}
\end{proof}

\section*{Conflict of Interest}
On behalf of all authors, the corresponding author states that there is no conflict of interest.
\section*{Data Availability}
Data sharing is not applicable to this article as no new datasets were generated or analyzed during the current study.


\begin{bibdiv}
\begin{biblist}   

\bib{AV19}{article}{
  author={Abatangelo, Nicola},
   author={Valdinoci, Enrico},
   title={Getting acquainted with the fractional Laplacian},
   conference={
      title={Contemporary research in elliptic PDEs and related topics},
   },
   book={
      series={Springer INdAM Ser.},
      volume={33},
      publisher={Springer, Cham},
   },
   date={2019},
   pages={1--105},
   review={\MR{3967804}},
}

\bib{BV16}{book} {
   author={Bucur, Claudia},
   author={Valdinoci, Enrico},
   title={Nonlocal diffusion and applications},
   series={Lecture Notes of the Unione Matematica Italiana},
   volume={20},
   publisher={Springer, [Cham]; Unione Matematica Italiana, Bologna},
   date={2016},
   pages={xii+155},
   isbn={978-3-319-28738-6},
   isbn={978-3-319-28739-3},
   review={\MR{3469920}},
   doi={10.1007/978-3-319-28739-3},
}

\bib{MR2644786}{article}{
   author={Cabr\'{e}, Xavier},
   author={Cinti, Eleonora},
   title={Energy estimates and 1-D symmetry for nonlinear equations
   involving the half-Laplacian},
   journal={Discrete Contin. Dyn. Syst.},
   volume={28},
   date={2010},
   number={3},
   pages={1179--1206},
   issn={1078-0947},
   review={\MR{2644786}},
   doi={10.3934/dcds.2010.28.1179},
}

\bib{MR3148114}{article}{
   author={Cabr\'{e}, Xavier},
   author={Cinti, Eleonora},
   title={Sharp energy estimates for nonlinear fractional diffusion
   equations},
   journal={Calc. Var. Partial Differential Equations},
   volume={49},
   date={2014},
   number={1-2},
   pages={233--269},
   issn={0944-2669},
   review={\MR{3148114}},
   doi={10.1007/s00526-012-0580-6},
}

\bib{MR4938046}{article}{
   author={Cabr\'{e}, Xavier},
   author={Cinti, Eleonora},
   author={Serra, Joaquim},
   title={Stable solutions to the fractional Allen-Cahn equation in the
   nonlocal perimeter regime},
   journal={Amer. J. Math.},
   volume={147},
   date={2025},
   number={4},
   pages={957--1024},
   issn={0002-9327},
   review={\MR{4938046}},
}

\bib{CS14}{article} {
  author={Cabr\'{e}, Xavier},
   author={Sire, Yannick},
   title={Nonlinear equations for fractional Laplacians, I: Regularity,
   maximum principles, and Hamiltonian estimates},
   journal={Ann. Inst. H. Poincar\'{e} C Anal. Non Lin\'{e}aire},
   volume={31},
   date={2014},
   number={1},
   pages={23--53},
   issn={0294-1449},
   review={\MR{3165278}},
   doi={10.1016/j.anihpc.2013.02.001},
}

\bib{MR3280032}{article}{
   author={Cabr\'{e}, Xavier},
   author={Sire, Yannick},
   title={Nonlinear equations for fractional Laplacians II: Existence,
   uniqueness, and qualitative properties of solutions},
   journal={Trans. Amer. Math. Soc.},
   volume={367},
   date={2015},
   number={2},
   pages={911--941},
   issn={0002-9947},
   review={\MR{3280032}},
   doi={10.1090/S0002-9947-2014-05906-0},
}

\bib{MR2177165}{article}{
   author={Cabr\'{e}, Xavier},
   author={Sol\`a-Morales, Joan},
   title={Layer solutions in a half-space for boundary reactions},
   journal={Comm. Pure Appl. Math.},
   volume={58},
   date={2005},
   number={12},
   pages={1678--1732},
   issn={0010-3640},
   review={\MR{2177165}},
   doi={10.1002/cpa.20093},
}



\bib{MR2354493}{article}{
    AUTHOR = {Caffarelli, Luis},
	AUTHOR= {Silvestre, Luis},
     TITLE = {An extension problem related to the fractional {L}aplacian},
   JOURNAL = {Comm. Partial Differential Equations},
  FJOURNAL = {Communications in Partial Differential Equations},
    VOLUME = {32},
      YEAR = {2007},
    NUMBER = {7-9},
     PAGES = {1245--1260},
      ISSN = {0360-5302,1532-4133},
   MRCLASS = {35J70},
  MRNUMBER = {2354493},
MRREVIEWER = {Francesco\ Petitta},
       DOI = {10.1080/03605300600987306},
       URL = {https://doi.org/10.1080/03605300600987306},
}


\bib{MR4612096}{article}{
   author={Conti, Sergio},
   author={Garroni, Adriana},
   author={M\"{u}ller, Stefan},
   title={Derivation of strain-gradient plasticity from a generalized
   Peierls-Nabarro model},
   journal={J. Eur. Math. Soc. (JEMS)},
   volume={25},
   date={2023},
   number={7},
   pages={2487--2524},
   issn={1435-9855},
   review={\MR{4612096}},
   doi={10.4171/jems/1242},
}

\bib{CP16}{article}{
author={Cozzi, Matteo},
   author={Passalacqua, Tommaso},
   title={One-dimensional solutions of non-local Allen-Cahn-type equations
   with rough kernels},
   journal={J. Differential Equations},
   volume={260},
   date={2016},
   number={8},
   pages={6638--6696},
   issn={0022-0396},
   review={\MR{3460227}},
   doi={10.1016/j.jde.2016.01.006},
}

\bib{CozziValdNONLINEARITY}{article}{
  author={Cozzi, Matteo},
   author={Valdinoci, Enrico},
   title={Planelike minimizers of nonlocal Ginzburg-Landau energies and
   fractional perimeters in periodic media},
   journal={Nonlinearity},
   volume={31},
   date={2018},
   number={7},
   pages={3013--3056},
   issn={0951-7715},
   review={\MR{3816747}},
   doi={10.1088/1361-6544/aab89d},
}

\bib{DPDV}{article} {
    AUTHOR = {De Pas, F.},
    AUTHOR = {Dipierro, S.},
    AUTHOR = {Piccinini, M.},
    AUTHOR = {Valdinoci, E.},
     TITLE = {Heteroclinic connections for fractional Allen-Cahn equations with degenerate potentials},
   JOURNAL = {Annali della Scuola Normale Superiore di Pisa - Classe di Scienze, Section: forthcoming articles},
  FJOURNAL = {},
    VOLUME = {},
      YEAR = {},
    NUMBER = {},
     PAGES = {},
      ISSN = {},
   MRCLASS = {},
  MRNUMBER = {},
MRREVIEWER = {},
       DOI = {10.2422/2036-2145.202502.001},
       URL = {https://journals.sns.it/index.php/annaliscienze/article/view/6978/2424},
}

\bib{OURREC}{article} {
    AUTHOR = {De Pas, F.},
    AUTHOR = {Dipierro, S.},
    AUTHOR = {Valdinoci, E.},
     TITLE = {Optimal decay of heteroclinic solutions of the fractional Allen-Cahn equation with a degenerate potential},
   JOURNAL = {preprint},
%%  FJOURNAL = {},
%%    VOLUME = {},
%%      YEAR = {},
%%    NUMBER = {},
%%     PAGES = {},
%%      ISSN = {},
%%   MRCLASS = {},
%%  MRNUMBER = {},
%%MRREVIEWER = {},
}

\bib{MR3740395}{article}{
   author={Dipierro, Serena},
   author={Farina, Alberto},
   author={Valdinoci, Enrico},
   title={A three-dimensional symmetry result for a phase transition
   equation in the genuinely nonlocal regime},
   journal={Calc. Var. Partial Differential Equations},
   volume={57},
   date={2018},
   number={1},
   pages={Paper No. 15, 21},
   issn={0944-2669},
   review={\MR{3740395}},
   doi={10.1007/s00526-017-1295-5},
}

\bib{DFV14}{article}{
author={Dipierro, Serena},
   author={Figalli, Alessio},
   author={Valdinoci, Enrico},
   title={Strongly nonlocal dislocation dynamics in crystals},
   journal={Comm. Partial Differential Equations},
   volume={39},
   date={2014},
   number={12},
   pages={2351--2387},
   issn={0360-5302},
   review={\MR{3259559}},
   doi={10.1080/03605302.2014.914536},
}

\bib{DPV15}{article}{
author={Dipierro, Serena},
   author={Palatucci, Giampiero},
   author={Valdinoci, Enrico},
   title={Dislocation dynamics in crystals: a macroscopic theory in a
   fractional Laplace setting},
   journal={Comm. Math. Phys.},
   volume={333},
   date={2015},
   number={2},
   pages={1061--1105},
   issn={0010-3616},
   review={\MR{3296170}},
   doi={10.1007/s00220-014-2118-6},
}

\bib{MR4531940}{article}{
  author={Dipierro, Serena},
   author={Patrizi, Stefania},
   author={Valdinoci, Enrico},
   title={A fractional glance to the theory of edge dislocations},
   conference={
      title={Geometric and Functional Inequalities and Recent Topics in
      Nonlinear PDEs},
   },
   book={
      series={Contemp. Math.},
      volume={781},
      publisher={Amer. Math. Soc., [Providence], RI},
   },
   date={2023},
   pages={103--135},
   review={\MR{4531940}},
   doi={10.1090/conm/781/15710},
}

\bib{MR4124116}{article}{
   author={Dipierro, Serena},
   author={Serra, Joaquim},
   author={Valdinoci, Enrico},
   title={Improvement of flatness for nonlocal phase transitions},
   journal={Amer. J. Math.},
   volume={142},
   date={2020},
   number={4},
   pages={1083--1160},
   issn={0002-9327},
   review={\MR{4124116}},
   doi={10.1353/ajm.2020.0032},
}

\bib{MR4297378}{article}{
   author={Dipierro, Serena},
   author={Savin, Ovidiu},
   author={Valdinoci, Enrico},
   title={On divergent fractional Laplace equations},
   language={English, with English and French summaries},
   journal={Ann. Fac. Sci. Toulouse Math. (6)},
   volume={30},
   date={2021},
   number={2},
   pages={255--265},
   issn={0240-2963},
   review={\MR{4297378}},
   doi={10.5802/afst.1673},
}

\bib{MR4581189}{article}{
   author={Dipierro, Serena},
   author={Valdinoci, Enrico},
   title={Some perspectives on (non)local phase transitions and minimal
   surfaces},
   journal={Bull. Math. Sci.},
   volume={13},
   date={2023},
   number={1},
   pages={Paper No. 2330001, 77},
   issn={1664-3607},
   review={\MR{4581189}},
   doi={10.1142/S1664360723300013},
}

\bib{MR4050103}{article}{
   author={Figalli, Alessio},
   author={Serra, Joaquim},
   title={On stable solutions for boundary reactions: a De Giorgi-type
   result in dimension $4+1$},
   journal={Invent. Math.},
   volume={219},
   date={2020},
   number={1},
   pages={153--177},
   issn={0020-9910},
   review={\MR{4050103}},
   doi={10.1007/s00222-019-00904-2},
}

\bib{MR2461827}{article}{
   author={Forcadel, Nicolas},
   author={Imbert, Cyril},
   author={Monneau, R\'{e}gis},
   title={Homogenization of some particle systems with two-body interactions
   and of the dislocation dynamics},
   journal={Discrete Contin. Dyn. Syst.},
   volume={23},
   date={2009},
   number={3},
   pages={785--826},
   issn={1078-0947},
   review={\MR{2461827}},
   doi={10.3934/dcds.2009.23.785},
}

\bib{GM12}{article}{
      author={Gonz\'{a}lez, Mar\'{\i}a del Mar},
   author={Monneau, Regis},
   title={Slow motion of particle systems as a limit of a reaction-diffusion
   equation with half-Laplacian in dimension one},
   journal={Discrete Contin. Dyn. Syst.},
   volume={32},
   date={2012},
   number={4},
   pages={1255--1286},
   issn={1078-0947},
   review={\MR{2851899}},
   doi={10.3934/dcds.2012.32.1255},
}

\bib{MR1829143}{article}{
   author={Haraux, A.},
   author={Jendoubi, M. A.},
   title={Decay estimates to equilibrium for some evolution equations with
   an analytic nonlinearity},
   journal={Asymptot. Anal.},
   volume={26},
   date={2001},
   number={1},
   pages={21--36},
   issn={0921-7134},
   review={\MR{1829143}},
   doi={10.3233/asy-2001-437},
}

\bib{MR1609269}{article}{
   author={Jendoubi, Mohamed Ali},
   title={A simple unified approach to some convergence theorems of L.
   Simon},
   journal={J. Funct. Anal.},
   volume={153},
   date={1998},
   number={1},
   pages={187--202},
   issn={0022-1236},
   review={\MR{1609269}},
   doi={10.1006/jfan.1997.3174},
}


\bib{MR371203}{book}{
   author={Lardner, R. W.},
   title={Mathematical theory of dislocations and fracture},
   series={Mathematical Expositions, No. 17},
   publisher={University of Toronto Press, Toronto, ON},
   date={1974},
   pages={xi+363},
   review={\MR{371203}},
}

\bib{MR2946964}{article}{
   author={Monneau, R\'{e}gis},
   author={Patrizi, Stefania},
   title={Homogenization of the Peierls-Nabarro model for dislocation
   dynamics},
   journal={J. Differential Equations},
   volume={253},
   date={2012},
   number={7},
   pages={2064--2105},
   issn={0022-0396},
   review={\MR{2946964}},
   doi={10.1016/j.jde.2012.06.019},
}

\bib{PSV13}{article}{
      author={Palatucci, Giampiero},
   author={Savin, Ovidiu},
   author={Valdinoci, Enrico},
   title={Local and global minimizers for a variational energy involving a
   fractional norm},
   journal={Ann. Mat. Pura Appl. (4)},
   volume={192},
   date={2013},
   number={4},
   pages={673--718},
   issn={0373-3114},
   review={\MR{3081641}},
   doi={10.1007/s10231-011-0243-9},
}

\bib{MR3338445}{article}{
   author={Patrizi, Stefania},
   author={Valdinoci, Enrico},
   title={Crystal dislocations with different orientations and collisions},
   journal={Arch. Ration. Mech. Anal.},
   volume={217},
   date={2015},
   number={1},
   pages={231--261},
   issn={0003-9527},
   review={\MR{3338445}},
   doi={10.1007/s00205-014-0832-z},
}

\bib{MR3511786}{article}{
   author={Patrizi, Stefania},
   author={Valdinoci, Enrico},
   title={Relaxation times for atom dislocations in crystals},
   journal={Calc. Var. Partial Differential Equations},
   volume={55},
   date={2016},
   number={3},
   pages={Art. 71, 44},
   issn={0944-2669},
   review={\MR{3511786}},
   doi={10.1007/s00526-016-1000-0},
}

\bib{MR3703556}{article}{
   author={Patrizi, Stefania},
   author={Valdinoci, Enrico},
   title={Long-time behavior for crystal dislocation dynamics},
   journal={Math. Models Methods Appl. Sci.},
   volume={27},
   date={2017},
   number={12},
   pages={2185--2228},
   issn={0218-2025},
   review={\MR{3703556}},
   doi={10.1142/S0218202517500427},
}


\bib{MR1680877}{article}{
   author={Rybka, Piotr},
   author={Hoffmann, Karl-Heinz},
   title={Convergence of solutions to Cahn-Hilliard equation},
   journal={Comm. Partial Differential Equations},
   volume={24},
   date={1999},
   number={5-6},
   pages={1055--1077},
   issn={0360-5302},
   review={\MR{1680877}},
   doi={10.1080/03605309908821458},
}


\bib{MR3812860}{article}{
   author={Savin, Ovidiu},
   title={Rigidity of minimizers in nonlocal phase transitions},
   journal={Anal. PDE},
   volume={11},
   date={2018},
   number={8},
   pages={1881--1900},
   issn={2157-5045},
   review={\MR{3812860}},
   doi={10.2140/apde.2018.11.1881},
}

\bib{MR3939768}{article}{
   author={Savin, O.},
   title={Rigidity of minimizers in nonlocal phase transitions II},
   journal={Anal. Theory Appl.},
   volume={35},
   date={2019},
   number={1},
   pages={1--27},
   issn={1672-4070},
   review={\MR{3939768}},
   doi={10.4208/ata.oa-0008},
}

      \bib{SV12}{article}{
 author={Savin, Ovidiu},
   author={Valdinoci, Enrico},
   title={$\Gamma$-convergence for nonlocal phase transitions},
   journal={Ann. Inst. H. Poincar\'{e} C Anal. Non Lin\'{e}aire},
   volume={29},
   date={2012},
   number={4},
   pages={479--500},
   issn={0294-1449},
   review={\MR{2948285}},
   doi={10.1016/j.anihpc.2012.01.006},
}

\bib{MR3035063}{article}{
   author={Savin, Ovidiu},
   author={Valdinoci, Enrico},
   title={Some monotonicity results for minimizers in the calculus of
   variations},
   journal={J. Funct. Anal.},
   volume={264},
   date={2013},
   number={10},
   pages={2469--2496},
   issn={0022-1236},
   review={\MR{3035063}},
   doi={10.1016/j.jfa.2013.02.005},
}

\bib{SV14}{article}{
   author={Savin, Ovidiu},
   author={Valdinoci, Enrico},
   title={Density estimates for a variational model driven by the Gagliardo
   norm},
   language={English, with English and French summaries},
   journal={J. Math. Pures Appl. (9)},
   volume={101},
   date={2014},
   number={1},
   pages={1--26},
   issn={0021-7824},
   review={\MR{3133422}},
   doi={10.1016/j.matpur.2013.05.001},
}

\bib{S07}{article} {
       author={Silvestre, Luis},
   title={Regularity of the obstacle problem for a fractional power of the
   Laplace operator},
   journal={Comm. Pure Appl. Math.},
   volume={60},
   date={2007},
   number={1},
   pages={67--112},
   issn={0010-3640},
   review={\MR{2270163}},
   doi={10.1002/cpa.20153},
}

\bib{MR2498561}{article}{
   author={Sire, Yannick},
   author={Valdinoci, Enrico},
   title={Fractional Laplacian phase transitions and boundary reactions: a
   geometric inequality and a symmetry result},
   journal={J. Funct. Anal.},
   volume={256},
   date={2009},
   number={6},
   pages={1842--1864},
   issn={0022-1236},
   review={\MR{2498561}},
   doi={10.1016/j.jfa.2009.01.020},
}

\bib{MR1442163}{article}{
   author={Toland, J. F.},
   title={The Peierls-Nabarro and Benjamin-Ono equations},
   journal={J. Funct. Anal.},
   volume={145},
   date={1997},
   number={1},
   pages={136--150},
   issn={0022-1236},
   review={\MR{1442163}},
   doi={10.1006/jfan.1996.3016},
}

\end{biblist}
\end{bibdiv}
\end{document}